\documentclass[a4paper, 10pt]{article}
\usepackage{amssymb}
\usepackage{amsmath}
\usepackage{mathrsfs}
\usepackage{amsfonts}
\usepackage{color}
\usepackage{vmargin}
\usepackage{amsthm}
\usepackage{ esint }
\usepackage{a4}
\usepackage{graphicx}
\usepackage{calligra}
\usepackage[T1]{fontenc}

\long\def\symbolfootnote[#1]#2{\begingroup%
\def\thefootnote{\fnsymbol{footnote}}\footnote[#1]{#2}\endgroup}

\setmarginsrb{20mm}{20mm}{20mm}{20mm}{10mm}{10mm}{10mm}{10mm}

\long\def\symbolfootnote[#1]#2{\begingroup%
\def\thefootnote{\fnsymbol{footnote}}\footnote[#1]{#2}\endgroup}

\def\vint{\mathop{\mathchoice%
          {\setbox0\hbox{$\displaystyle\intop$}\kern 0.22\wd0%
           \vcenter{\hrule width 0.6\wd0}\kern -0.82\wd0}%
          {\setbox0\hbox{$\textstyle\intop$}\kern 0.2\wd0%
           \vcenter{\hrule width 0.6\wd0}\kern -0.8\wd0}%
          {\setbox0\hbox{$\scriptstyle\intop$}\kern 0.2\wd0%
           \vcenter{\hrule width 0.6\wd0}\kern -0.8\wd0}%
          {\setbox0\hbox{$\scriptscriptstyle\intop$}\kern 0.2\wd0%
           \vcenter{\hrule width 0.6\wd0}\kern -0.8\wd0}}%
          \mathopen{}\int}

\newcommand{\Om}{\Omega}
\newcommand{\C}{\mathbb{C}}
\newcommand{\R}{\mathbb{R}}

\newcommand{\Hei}{{\mathbb{H}}_{1}}

\newcommand{\HW}{{HW}^{1,s}}
\newcommand{\HWtw}{{HW}^{1,2}}
\newcommand{\HWtwloc}{{HW}^{1,2}_{loc}}

\newcommand{\HWzero}{{HW}^{1,s}_{0}}

\newcommand{\ga}{\gamma}

\newcommand{\bZ}{{\bar Z}}

\definecolor{blau}{rgb}{0.1,0.0,0.9}

\definecolor{funk}{rgb}{0.1,0.4,0.9}

\newcounter{komcounter}
\numberwithin{komcounter}{section}

\newcommand{\dbd}[2]{\frac{\partial#1}{\partial #2}}

\def\XXint#1#2#3{{\setbox0=\hbox{$#1{#2#3}{\int}$}
     \vcenter{\hbox{$#2#3$}}\kern-.5\wd0}}

\theoremstyle{plain}
\newtheorem{theorem}{Theorem}[section]
\newtheorem{thm}{Theorem}[section]
\newtheorem{lem}[theorem]{Lemma}
\newtheorem{cor}[theorem]{Corollary}
\newtheorem{prop}[theorem]{Proposition}
\newtheorem{observ}[theorem]{Observation}

\theoremstyle{definition}
\newtheorem{defn}[theorem]{Definition}
\newtheorem{ex}{Example}
\newtheorem{rem}{\textnormal{\textbf{Remark}}}

\begin{document}

\title {Schwarzians on the Heisenberg group}

\author{
Tomasz Adamowicz{\small{$^*$}}
\\
\it\small Institute of Mathematics, Polish Academy of Sciences \\
\it\small ul. \'Sniadeckich 8, 00-656 Warsaw, Poland\/{\rm ;}
\it\small T.Adamowicz@impan.pl
\\
\\
Ben Warhurst
\\
\it\small Institute of Mathematics,
\it\small University of Warsaw,\\
\it\small ul. Banacha 2, 02-097 Warsaw, Poland\/{\rm ;}
\it\small B.Warhurst@mimuw.edu.pl
}

\date{}
\maketitle

\footnotetext[1]{T. Adamowicz and B. Warhurst were supported by the grant of the National Science Center, Poland (NCN), UMO-2017/25/B/ST1/01955.}

\begin{abstract}
 We study various notions of the Schwarzian derivative for contact mappings in the Heisenberg group $\Hei$ and introduce two definitions: (1) \emph{the CR Schwarzian derivative} based on the conformal connection approach studied by Osgood and Stowe and, recently, by Son; (2) \emph{the classical type Schwarzian} refering to the well-known complex analytic definition. In particular, we take into consideration the effect of conformal rigidity and the limitations it imposes. Moreover, we study the kernels of both Schwarzians and the cocycle conditions. Our auxiliary results include a characterization of the contact conformal vector fields.

 Inspired by ideas of Chuaqui--Duren--Osgood~\cite{CDO}, Hern\'andez, Mart\'in and Venegas~\cite{HM, HV}, we introduce the \emph{Preschwarzian} for mappings in $\Hei$. Furthermore, we study results in the theory of subelliptic PDEs for the horizontal Jacobian and related differential expressions for harmonic mappings and the \emph{gradient harmonic mappings}, the latter notion introduced here in the setting of $\Hei$.
\newline
\newline \emph{Keywords}: Carnot group, conformal map, harmonic map, Heisenberg group, Jacobian, Laplacian, Lie algebra, preschwarzian, schwarzian, subsolution, supersolution.
\newline
\newline
\emph{Mathematics Subject Classification (2020):} Primary: 53A30; Secondary: 32V20, 35H20, 35R03
\end{abstract}

\section{Introduction} The Schwarzian has many interpretations and contexts, see Valentin Ovsienko and Sergei Tabachnikov \cite{VTS} for a survey. The context we will be focusing on in this discussion is with regard to conformal mappings. In \cite{VTS} they give the following algorithm for considering multi-dimensional versions of the Schwarzian: a) choose a group of diffeomorphisms and a
subgroup $G$ that has a nice geometrical meaning,
b) find a $G$-invariant 1-cocycle on the group of
diffeomorphisms,
c) (the most important step) check that no one did
it before. One of the results of this paper suggests that this algorithm should have also have a part d) which says: consider the consequences of rigidity properties of the group $G$.

Recall that the Schwarzian derivative of a locally conformal mapping $f$ on a domain in the complex plane is given by
\begin{align}
S(f)&= \frac{f'''}{f'}-\frac{3}{2} \left(\frac{f''}{f'}\right)^2. \label{CplainDef}
\end{align}
 Its vanishing distinguishes M\"obius transformations amongst local conformal mappings. An  important property of the Schwarzian is the composition formula
 $$
 S(f \circ g)= S(g) + (S(f) \circ g)(g')^2.
 $$
 It then follows that the Schwarzian tensor
 $$
 \Sigma(f)(z) = S(f)(z) dz \otimes dz
 $$
 satisfies the following cocycle condition
\begin{align} \Sigma(f \circ g) = \Sigma(g) +g^*\Sigma(f). \label{cocycletenscond} \end{align}
Fundamental to the foundations for the Schwarzian is the fact that locally conformal mappings form an infinite dimensional pseudo group. The abundance of conformal maps means that there is a distinction between local and global,  which is precisely what the vanishing of the Schwarzian detects.

If $f:(M, g) \to (M', g') $ is a locally conformal map between Riemannian manifolds of dimension $n \geq 2$, such that $f^*g'=e^{2\phi} g$, then Osgood and Stowe in \cite{OsgStow},  give the following definition:
$$
 \Sigma(f)=B_g( \phi) = \text{Hess}(\phi) - d\phi \otimes d \phi  -\frac{1}{n} ( \Delta \phi -\|\text{grad } \phi\|^2)g,
$$
 which by design, satisfies the cocycle condition \eqref{cocycletenscond} and agrees with the classical definition of $\Sigma(f)$ in dimension $2$.

In dimension $n>2$, the assumption that $f$ is locally conformal can make this Riemannian definition redundant if the geometries are conformally rigid.  Liouville type theorems can imply that a locally conformal map $f$ is the restriction of a globally conformal map, and $f$ is given by the action of a finite dimensional Lie group of automorphisms. In such a case, there is no distinction between local and global.

In the setting of CR manifolds, Son \cite{DuongNgocSon} gives a similar definition to that of Osgood and Stowe, aimed at characterising CR-M\"obius transformations between CR manifolds. For the Heisenberg group $\Hei$, the construction is as follows. If $\theta$ is the standard contact form on the Heisenberg group $\Hei$, then  a local diffeomorphism $f : U \subseteq \Hei \to \Hei$ is a contact mapping if and only if $f^*\theta=e^{2 \phi} \theta$ for some real valued function $\phi$. At a minimum, a M\"obius transformation should be a local contact diffeomorphism. Son defines the Schwarzian tensor of $f$ by $ \Sigma(f)=B_\theta(\phi)$ where $B_\theta(\phi)$ is the two form on $T^{1,0}\Hei \oplus T^{0,1}\Hei$ given by
\begin{align}
 B_\theta ( \phi) =  \text{Sym}\nabla^2 \phi- 4(\partial_\flat \phi \otimes \partial_\flat \phi  +  \bar \partial_\flat \phi \otimes \bar \partial_\flat \phi ) - ( \Delta_\flat \phi) L_\theta, \label{CRSchwarz}
\end{align}
 where $\nabla$ is the Tanaka--Webster connection, $\text{Sym}\nabla^2$ is twice the symmetrisation of $\nabla^2$, $\partial_\flat $ and $\bar \partial_\flat$ are the Dolbeault operators,  $ \Delta_\flat $ is the sub-Laplacian, and $  L_\theta$ is the Levi form. Although this definition makes sense for any local contact diffeomorphisms, it does not satisfy the cocycle condition \eqref{cocycletenscond} for such a wide class, and requires the additional restriction that the map is CR holomorphic in its complex coordinate when $\Hei$ is modeled on $\C \times \R$. The CR holomorphicity condition is the analogue of the map being holomorphic in the classical case. The definition of a M\"obius transformation is then given within the pseudogroup of CR contact diffeomorphisms as any such map $f$ that satisfies $ \Sigma(f) = 0$ and $\phi$ is CR-pluriharmonic (the real part of a CR holomorphic function). In the higher dimensional Heisenberg groups, the definition is the same, except that the assumption that $\phi$ is CR-pluriharmonic is removed, since according to \cite{DuongNgocSon} it can be shown to be a consequence of $ \Sigma(f) = 0$ which is not the case for $\Hei$.

The conformal rigidity of the Heisenberg group is an issue. More specifically, within the confines of the contact pseudo group, the three conditions: conformal, CR holomorphic in the complex coordinate and $f \in SU(1,2)$, are synonymous. Furthermore, the function $\phi$ is CR-pluriharmonic for all $f \in SU(1,2)$. So at the outset, the definition in \cite{DuongNgocSon} only applies to $f \in SU(1,2)$ and the conditions $ \Sigma(f) = 0$ and $\phi$ is CR-pluriharmonic are vacuous. For the Heisenberg groups $\mathbb{H}^n$, $n>1$, the story is much the same, the assumption that $f$ is contact and CR holomorphic in its complex coordinates implies $f \in SU(1,2n)$ and CR-pluriharmonicity of $\phi$ is a consequence of rigidity.

In this paper we present the details supporting the discussion above and investigate two types of Schwarzians resembling  \eqref{CplainDef}. First, motivated by the tensor~\eqref{CRSchwarz}, in Definition~\ref{def-cr-schw} we introduce the \emph{CR Schwarzian derivative} of a positively oriented contact mapping $f:\Om\to \Hei$ on an open set $\Om\subset \Hei$ defined as
\[
 S_{CR}(f)=\frac{1}{2} \lambda_{f } Z^2(\lambda_f^{-1}),
\]
where $\lambda_f=J_F$ is the horizontal Jacobian of $f=(f_1, f_2, f_3)$, and $F=f_1+i f_2$, see the presentation following Definition~\ref{def-contact}. As is customary, we will also consider $F$ to be a function taking values in $\R^2$.  In {\bf Section 5} we show that the class of contact mappings for which $S_{CR}$ vanishes contains the class of conformal mappings. Furthermore, it turns out that the kernel of $S_{CR}$ contains an infinite dimensional class of mappings that need not be closed under composition, see Observation~\ref{obs-CR-Schw-chainr} and its corollaries.

In {\bf Section 6} we introduce and study the \emph{classical type Schwarzian}, see Defintition~\ref{defn-cl-Schw}. If
 $f: \Omega \to \mathbb{H}^1$ is a contact mapping, we write $f=(f_1+if_2,f_3)=(F,f_3)$ and define
\begin{align}
S_{CL}(f)= Z \left (\frac{ Z^2 F}{ZF} \right )-\frac{1}{2} \left(\frac{Z^2F}{ZF}\right)^2= \frac{ Z^3 F}{ZF}-\frac{3}{2} \left(\frac{Z^2F}{ZF}\right)^2. 
\end{align}
at points where $ZF\not=0$. This classical type Schwarzian agrees with \cite{DuongNgocSon} when $f$ is conformal since it also vanishes for $f \in SU(1,2)$, see Theorem~\ref{thm61}. On the other hand, there are mappings $f$ such that $ S_{CR}(f)=0$ while $ S_{CL}(f) \ne 0$. In particular, the family of contact diffeomorphisms that are annihilated by the CR Schwarzian contains an infinite dimensional space arising as flows, most of which are not zeros for the classical type definition above, see  Example~\ref{ex3}. Moreover, the "right" cocycle condition holds for this classical Schwarzian when we change variable by an element of $SU(1,2)$, but it is not invariant when we normalise the values of $f$ by elements of $SU(1,2)$, as proven in Lemma~\ref{mycocycle}.

Finally, in {\bf Section 7} we introduce a notion of the \emph{Preschwarzian} in the Heisenberg group $\Hei$ for $C^1$ mapping with a positive horizontal Jacobian $\lambda_f=J_F>0$
\[
Pf:=Z\ln J_F.
\]
Similar notions have been studied, for instance, in the setting of harmonic mappings in the plane, see~\cite{HM}. In particular,~\cite{HM} shows for Preschwarzians several natural counterparts of results known for Schwarzians. The definition of the Preschwarzian relies on $J_F$, and therefore we will focus our attention on studying its properties and related expressions. As noticed by several authors, e.g. \cite{acs, HM, HV}, the Preschwarzian turns out to be a handy tool beyond the class of conformal mappings, especially for harmonic and quasiconformal mappings. Following this program,  we investigate super- and subharmonicity results for the class of harmonic mappings and the \emph{gradient harmonic mappings} (Definition~\ref{grad-harm}), the latter one introduced here in $\Hei$. Moreover, we relate gradient harmonic mappings to subelliptic PDEs and discover connections to the analysis of level sets (see the introductory paragraph in Section 7 for details).

The class of gradient harmonic mappings has not been investigated in the Heisenberg setting so far and we hope that our preliminary results for such mappings will draw more attention to their studies.

One of the key tools involved in the studies of both Schwarzians in the class of conformal mappings are the auxiliary results presented in {\bf Section 4}. In particular, we characterize contact conformal vector fields in terms of solutions to the system of equation $ZZ=0$, see Lemma~\ref{vzerosol} and its proof in the Appendix.

We should remark that there is another perspective on generalising the Schwarzian arising in the geometric theory of ordinary differential equations. In particular, there is a notion of the Schwarzian of a contact diffeomorphism of the Heisenberg group, related to classifying the family of third order equations $y'''=h(x,y,y',y'')$ that are equivalent to the homogeneous equation $y'''=0$ via the prolonged action of a contact diffeomorphism $f$ to the second jet space. The result is a system of four second order partial differential equations involving the coordinate functions of $f$, see \cite{Sato}. Ostensibly, this definition is not concerned with conformal mapping and is not related to the classical Schwarzian in an obvious geometric sense, hence we have not incorporated these results in our discussion.

\section{Preliminaries}

The Heisenberg group $\Hei$, modeled on $\mathbb{C} \times \mathbb{R}$, is given in coordinates $(z,t)$, where $z=x+iy \in \C$ and $t \in \R$,  by the multiplication rule
\begin{align}
(z_1,t_1)(z_2,t_2)&=(z_1+z_2, t_1+t_2 + 2\, {\rm Im}\,(z_1 \bar z_2)) \nonumber\\
&=(x_1 + x_2 , y_1 + y_2, t_1+t_2 + 2(x_2y_1-x_1y_2)). \label{Hei-mult}
\end{align}
We observe that $(z,t)^{-1}=(-z,-t)$ and denote the left translation $(z,t) \to (z_1,t_1)(z,t)$ by $\tau_{(z_1,t_1)}$. A basis for the left invariant vector fields is given by the following vector fields:
\begin{align}
 X = \dbd{}{x}+2y\dbd{}{t}, \quad   Y =\dbd{}{y}-2x\dbd{}{t} \quad {\rm and } \quad   T= \dbd{}{t}, \label{left-invariant-fields}
\end{align} with corresponding dual forms:
\begin{align*} dx, \quad dy \quad  \text{and} \quad   \theta = dt- 2y dx+ 2x dy.
\end{align*}
The left invariant fields \eqref{left-invariant-fields} have the one nontrivial bracket given by $[  X,  Y] = -4  T$ which ensures that any two points $p$ and $q$ in $\mathbb{H}^1$ can be connected by an absolutely continuous curve $\gamma:[0,1]\to\mathbb{H}^1$ with the property that
\begin{displaymath}
\dot{\gamma}(s)\in  \mathcal{H}_{\gamma(s)},\quad \text{a.e. }s\in [0,1],\text{ where } \mathcal{H}_p:= \mathrm{span}\{ X_p, Y_p\}.
\end{displaymath}
Such a $\gamma$ is called a \emph{horizontal curve}.

The \emph{sub-Riemannian distance} $d_s$ is defined by
\begin{align}
d_s(p,q)=\inf_{\gamma} \int_0^1 \sqrt{\dot{\gamma}_1(s)^2 + \dot{\gamma}_2(s)^2}\;\mathrm{d}s,
\label{submet} \end{align}
where the infimum is taken over all horizontal curves $\gamma=(\gamma_1,\gamma_2,\gamma_3):[0,1]\to \mathbb{H}^1$ that connect $p$ and $q$. It is well known that $d_s$ defines a left-invariant metric on $\mathbb{H}^1$ which is homogeneous under the \emph{Heisenberg dilations}  $(\delta_{\lambda})_{\lambda>0}$, given by
\begin{displaymath}
\delta_{\lambda}:\mathbb{H}^1 \to \mathbb{H}^1,\quad \delta_{\lambda}(z,t)=(\lambda z,\lambda^2 t).
\end{displaymath}
Any two homogeneous left-invariant metrics on $\mathbb{H}^1$ are bi-Lipschitz equivalent, and it is often more convenient to work with a metric which is given by an explicitly computable formula, rather than to use $d_s$. An example of such a metric is the \emph{Kor\'{a}nyi distance}, defined by
\begin{displaymath}
d_{\mathbb{H}^1}(p,q) = \|q^{-1}p\|_{\mathbb{H}^1},\quad \text{where }\,\,\|(x,y,t)\|_{\mathbb{H}^1}= \sqrt[4]{|z|^4 + t^2}.
\end{displaymath}
The length distance associated to $d_{\mathbb{H}^1}$ is exactly $d_s$.


\begin{defn}\label{def-contact}
 Let $\Om,\Om'\subset \Hei$ be domains in $\Hei$. We say that a local $C^1$ diffeomorphism $f=(f_1+i f_2, f_3):\Om\to\Om'$ is a \emph{contact mapping/transformation} if it preserves the contact structure, i.e.,
	\begin{equation}
	f^{*}\theta = \lambda_f \theta, \label{contact0}
	\end{equation}
	where $\lambda_f:\Om \to \R$ and $\lambda_f \not= 0$ in $\Om$.
\end{defn}
Condition~\eqref{contact0} can be equivalently stated in terms of \emph{the contact equations}:
\begin{equation}\label{contact-eqs}
Xf_3-2f_2Xf1+2f_1Xf_2=0,\quad Yf_3-2f_2Yf1+2f_1Yf_2=0.
\end{equation}
We would like to point out that in what follows we also study maps satisfying the contact equations but not necessarily local diffeomorphism assumption.

Note that the definition implies that $Df$ preserves the horizontal bundle, and
relative to the left invariant basis, we have
$$
Df = \begin{bmatrix}
Xf_1 & Yf_1  & T f_1 \\
Xf_2 & Yf_2  & T f_2 \\
0 & 0 & \lambda_f
\end{bmatrix}. $$
We set $$
D_Hf=\begin{bmatrix}
Xf_1 & Yf_1 \\
Xf_2 & Yf_2
\end{bmatrix}
$$
and note that Cartan's formula together with $ f^*d \theta = d \lambda_f \wedge \theta + \lambda_f d \theta $ imply that $\lambda_f=\det D_Hf:=J_F$ and $\det Df= (\lambda_f)^2$. A contact map for which $\lambda_f>0$ will called positively oriented. Notice that this notion differs from an orientation preserving mapping, since $\det Df>0$ need not imply that $\lambda_f>0$.

Finally, we recall that the contact equations~\eqref{contact-eqs} imply that
\begin{equation}\label{contact-jac}
 J_F=Tf_3-2f_2Tf_1+2f_1Tf_2.
\end{equation}

\section{Conformal Transformations}
In this section we recall the basic definition of a sub-Riemannian conformal map on $\Hei$, and the  properties these mappings have relative to the CR-structure of $\Hei$.

The sub-Riemannin distance defined in \eqref{submet}, derives from the sub-Riemannian metric tensor which on each $\mathcal{H}_p$ is determined by the requirement that $\{X_p,Y_p\}$ is orthonormal. We denote the sub-Riemannian metric tensor at $p$ by $\langle \, , \, \rangle_p$.
\begin{defn} A contact transformation $f$ of a domain $\Omega \subseteq \Hei$, is said to be conformal on $\Omega$ if it preserves the angles between intersecting horizontal curves in $\Omega$. More precisely,
$$
\langle Df(p)V , Df(p)W\, \rangle_p = e^{2 \phi(p)} \langle V , W\, \rangle_p \quad \hbox{for all }\, \, V,W \in \mathcal{H}_p
$$
or equivalently 	
\begin{equation}
( X f_1)^2+ ( X f_2)^2  =  ( Y f_1)^2+ ( Y f_2)^2 = e^{2 \phi} \quad {\rm and} \quad  ( X f_1)( Y f_1) +  ( X f_2)( Y f_2) =0. \label{confmap1}
\end{equation}
\end{defn}


Conformal transformations in the Heisenberg group $\Hei$ are $C^\infty$, see \cite{CapCow}. In particular, the class of local contact diffeomorphisms can be assumed to be $C^\infty$ for most intents and purposes surrounding conformal mapping theory.

The CR-structure on $\mathbb{H}^1$ is that of the $3$-sphere in $\C^2$ and given pointwise by
$$
T^{1,0} \Hei =  \text{span}_\C \left \{ Z=\frac{1}{2}(X-i Y) \right \} \quad  {\rm and} \quad T^{0,1} \Hei  =  \text{span}_\C \left \{ \bar Z=\frac{1}{2}(X+i Y) \right \}.
$$

Recall the following observation characterizing conformal mappings as CR holomorphic functions. The following lemma is known to hold for all Carnot groups which carry a CR structure, see \cite{CowLiOtazWu}.

\begin{lem} \label{confZbar}
A positively oriented local $C^1$ contact diffeomorphism $f=(F,f_3)$ is conformal, if and only if $\bar Z F=0$.
\end{lem}

\begin{proof}
If $F=f_1+if_2$ then $ \lambda_f = |ZF|^2-|\bar Z F|^2$ and so the assumption $\lambda_f>0$ requires $ZF \ne 0$ and $ |ZF| > |\bar Z F|$. Moreover, the contact conditions become
\begin{align} Zf_3 = f_2 Zf_1 -f_1 Zf_2=\frac{i}{2} \Big ( \bar F ZF - F Z\bar F   \Big ). \label{zcontact}
\end{align}  Furthermore
\begin{align}
4 ZF Z \bar F &= (Zf_1)^2  + (Zf_2)^2 \nonumber\\
& =  ( X f_1)^2+ ( X f_2)^2  -  ( Y f_1)^2 - ( Y f_2)^2 - 2i \left ( ( X f_1)( Y f_1) +  ( X f_2)( Y f_2) \right )=0 \label{confCR}
\end{align}
and so it follows that a positively oriented conformal map satisfies $\bar Z F=0$, i.e., $F$ is a CR holomorphic function.

 Conversely, if $f$ is a positively oriented contact map and $\bar Z F=0$, then \eqref{confCR} implies that $f$ satisfies \eqref{confmap1} and is thus conformal.
\end{proof}

A Liouville type theorem for $\Hei$ follows from \cite{CapCow} and Theorem 8 in \cite{kr1}, and states that a  conformal transformations of a domain $\Omega \subseteq \Hei$ is given by the action of $SU(1,2)$ restricted to $\Omega$. In particular, the following lemma gives the explicit description of the action.
\begin{lem}[See \cite{kr1}] \label{conftypes} A conformal map decomposes into the $5$ following explicit types:
	\begin{enumerate}
		\item Left translation $\tau_p$,
		\item Dilation $\delta_r$,
		\item Rotation: $r_\varphi(z,t) = (e^{i \varphi} z, t)$,
		\item Inversion: $\iota(z,t)= \frac{-1}{|z|^4 +t^2 } ( z(|z|^2 +it) , t)=  ( \frac{z}{it -|z|^2  }, \frac{-t}{|z|^4 +t^2 })$,
		\item Reflection: $(z,t) \to (\overline{z},-t)$.
\end{enumerate}\end{lem}

 The maps $1$ to $4$ preserve orientation and generate $SU(1,2)$.  In particular, every  $ g \in SU(1,2)$ has a decomposition either as
\begin{equation}\label{types12}
 \text{type 1: }\tau_p \circ r_\varphi \circ \delta_r \circ \tau_q \quad  {\rm or} \quad  \text{type 2: }\tau_p  \circ \iota \circ r_\varphi  \circ \delta_r \circ \tau_q
\end{equation}
 depending on $\infty$ being a fixed point or not, see \cite{kr1}.

For the higher dimensional Heisenberg groups $\mathbb{H}^n$ modeled on $\C^n \times \R$, the story is much the same, i.e., a conformal map decomposes into the $4$ following explicit types: left translation, dilation, $SU(n)$ and inversion. The inversion is given by the same formulas as in item 4 where $|z|^2= |z_1|^2 + \dots + |z_n|^2$ and the reflection is defined by item 5 but with $z \in \C^n$. See \cite{kr2} for details.

%


\section{Conformal Vector Fields}

Since the Schwarzians we are considering are third order operators, it will be useful to know some fundamental indentities for conformal maps at order three. Such identities arise naturally from the fact that the induced action of a conformal map on vector fields must preserve conformal vector fields. More precisely, a conformal vector field $V$ is one that generates a flow of conformal maps. So if $g_t$ is the flow of $V$, then by definition $h_t=f \circ g_t \circ f^{-1}$ is the flow of $Df(V)$, and obviously $h_t$ is conformal for a conformal $f$, so $Df(V)$ is conformal. This approach will reveal in some organised way, the fundamental higher order PDE that a conformal map should satisfy.

Let
 $$
 V= v_1{X} +v_2 Y-4v_0{T}
 $$
 be a contact vector field, i.e., the flow of $V$ consists of contact diffeomorphisms, or equivalently, $[V,\mathcal{H}] \subseteq \mathcal{H}$. The bracket condition implies that
\begin{align*}
{X} v_0=-v_2 \quad {\rm and} \quad   Y v_0=v_1.
\end{align*}
%

If $V$ generates a flow of conformal mappings, then the conditions \eqref{confmap1} imply that the following conditions must be satisfied:
\begin{align}
{X}v_1 =  Y v_2  \quad {\rm and} \quad   Y v_1 = -{X} v_2. \label{subconf1}
\end{align}
Thus contact and conformality require that the following are satisfied:
\begin{align}
{X}^2v_0= Y^2v_0  \quad {\rm and} \quad  {X}  Y v_0= - Y {X}v_0. \label{ContConf}
\end{align}

We are in a position to state the key observation of this section.
\begin{lem} \label{vzerosol}
The following conditions are equivalent:
\begin{itemize}
\item[(1)] A contact vector field $V =v_1{X} +v_2 Y-4v_0{T}$ is conformal.
\item[(2)] The equations \eqref{ContConf} hold if and only if $Z^2v_0=0$ and have the $8$ real parameter solution
$$
v_0(x, y, t) =  c_1(x^4+2 x^2 y^2 +y^4+t^2)+ c_2(ty-xy^2-x^3) + c_3(tx + x^2y + y^3) +  c_4(x^2+y^2)+ c_5x + c_6y + c_7t + c_8.
$$
\end{itemize}
\end{lem}
The proof of the lemma is elementary, however the computations are tedious and therefore relegated to the Appendix.

Let $f$ be a contact map. Denote by $h=f^{-1}$ and by $q=h(p)$. Then we have that
\begin{align*}
W_q=(h_*V)_{q}  &= \left (  X h_1(p)v_1(p) +   Y h_1(p)v_2(p) -4 Th_1(p)v_0(p) \right )  X_{q}\\
& \quad + \left (  X h_2(p)v_1(p) +   Y h_2(p)v_2(p) -4 Th_2(p)v_0(p) \right )  Y_q\\
& \quad -4  \lambda_h(p) v_0(p)  T_q \\
& = w_1(q)  X_q +w_2(q)  Y_q -4w_0(q)  T_q,
\end{align*}
where
$$
w_0(h(p)) =\lambda_h(p) v_0(p).
$$
Since $p=f(q)$ we get $w_0(q) = \lambda_h(f(q)) v_0(f(q))$, and then use $Dh \circ f =Df^{-1}$ to get
$w_0(q) = \lambda_f(q)^{-1} v_0(f(q))$. 
Applying Lemma~\ref{vzerosol} to the right hand side of the previous equation, we get
\[
 w_0(q)=\lambda_f^{-1}(f_3^2 + (f_1^2+f_2^2)^2)c_1+\lambda_f^{-1}(f_3 f_2-f_1(f_2^2 +f_1^2))c_2+\cdots+
 \lambda_f^{-1}c_8.
\]
There are then $8$ fundamental cases, with the $i$-th case determined by letting $c_j=0$ for all $j\not=i$ and $i=1,\ldots, 8$:
\begin{align}
1) \quad c_1 :  w_0&= \lambda_f^{-1}(f_3^2 + (f_1^2+f_2^2)^2) \nonumber \\
2) \quad c_2 :  w_0&= \lambda_f^{-1}(f_3 f_2-f_1(f_1^2 +f_2^2)) \nonumber \\
3) \quad c_3 :  w_0&= \lambda_f^{-1}(f_3 f_1+f_2(f_1^2 + f_2^2) ) \nonumber \\
4) \quad c_4 :  w_0&= \lambda_f^{-1}(f_1^2+f_2^2) \label{confvfieldlist}  \\
5) \quad c_5 :  w_0&= \lambda_f^{-1} f_1 \nonumber  \\
6) \quad c_6 :  w_0&= \lambda_f^{-1} f_2 \nonumber \\
7) \quad c_7 :  w_0&= \lambda_f^{-1} f_3 \nonumber \\
8) \quad c_8 :  w_0&= \lambda_f^{-1}. \nonumber
\end{align}
If $f$ is conformal, then the vector field $W=h_*V$ is conformal provided that $V$ is conformal, and so $w_0$ must satisfy $Z^2 w_0=0$. Hence, it follows that the cases above give 8 third order PDEs that a conformal map $f$ must satisfy. Conversely, we have the following Lemma.

\begin{lem} If a contact map $h$ satisfies $Z^2 w_0=0$ in the cases 4), 5), 6) and 8) given at \eqref{confvfieldlist}, then $F=f_1 + i f_2$ satisfies $ZF Z\bar F=0$. Moreover, if $f$ is positively oriented then $\bar ZF=0$ and $f$ is conformal.
\end{lem}

\begin{proof} Cases 5), 6) and 8) imply that
	\begin{align*}
	0 &= Z^2 (\lambda_f^{-1} F) = 2 Z(\lambda_f^{-1})ZF + \lambda_f^{-1}Z^2F\\
	0 &= Z^2 (\lambda_f^{-1} \bar F)=2 Z(\lambda_f^{-1})Z \bar F + \lambda_f^{-1}Z^2 \bar F.
	\end{align*}	
	Calculating from case 4) by using the identities above we get	
	\begin{align*}
	0= Z^2  (\lambda_f^{-1} |F|^2) &=  Z^2 (\lambda_f^{-1} F)  \bar F +  Z^2 (\lambda_f^{-1} \bar F) F + 2 \lambda_f^{-1} ZF Z\bar F.
	\end{align*}
\end{proof}
In view of the previous lemma, the cases 1), 2), 3) and 7) in \eqref{confvfieldlist} can be accounted for as consequences of the contact conditions.

\section{The CR Schwarzian on $\Hei$}

In this section we look closely at the CR Schwarzian discussed above at~\eqref{CRSchwarz}. Basing on this notion we define the \emph{CR Schwarzian derivative} for contact mappings in $\Hei$, see Definition~\ref{def-cr-schw}. In particular, we show that the class of contact mappings for which the CR Schwarzian derivative vanishes contains the class of conformal mappings, see
the presentation following Definition~\ref{def-cr-schw} and also Examples~\ref{ex1} and~\ref{ex1SL2R}. Moreover, the nullset of the CR Schwarzian derivative contains (an infinite dimensional) class of mappings that need not be closed under composition, see Observation~\ref{obs-CR-Schw-chainr} and its consequences in Claims 1-3. To this end we begin by recalling some basic concepts from CR geometry.

The triple $(\Hei, T^{1,0} \Hei, \theta)$ is a strictly pseudo-convex pseudo-hermitian manifold of CR dimension one in the sense of Webster [30]. A complex structure $J$ is given by $J  X= Y$ and $J  Y= - X$ and extending linearly. It then follows that $JZ=iZ$ and $J\overline{Z}=i \bar Z$. Furthermore, the Reeb field is given by $T =\dbd{}{t}= \frac{i}{2}[Z, \bar Z]$ and the corresponding dual forms are $$ dz \, = \, dx+i dy, \quad d \bar z\, = \, dx-idy,\quad   \theta \, = \, dt -i \bar z dz +i z d \bar z,$$ and $$ d \theta =4 dx \wedge dy = 2i dz \wedge d\bar z.$$
The Levi-form is given by
$$
L_\theta (V_1,  V_2 ) = d \theta (V_1, J  V_2 ).
$$
 We note that the Levi form defines a sub-Riemmanian distance which differs from \eqref{submet} by the dilation $\delta_2$, hence the conformal transformations relative to the Levi form are those given in \eqref{conftypes}.



Moreover, in order to further align our notation with \cite{DuongNgocSon}, we note that for a $C^1$ complex valued function $F$ in $\Hei$, we have
\begin{align*}
dF &= (ZF) dz +  (\overline{Z} F) d \bar z + (TF) \theta:= \partial_\flat F  + \bar \partial_\flat F +(TF) \theta.
\end{align*}
If $g=(G,g_3)$, $G=g_1+i g_2$, is a local diffeomorphism of $\Hei$, then the chain rule takes the form
\begin{align}
 V(F \circ g) = g^*dF(V)= (ZF \circ g) VG + (\bar Z F \circ g) V \bar G + (TF \circ g) g^*\theta(V),\label{ComplexChainRule}
 \end{align}
 where $V$ is any vector field with values in $\C \otimes T\Hei$, $F$ is a $C^1$ complex valued function $F$ and $g$ is a local diffeomorphism. Formula \eqref{ComplexChainRule} will be used extensively for obvious reasons.  An important case occurs when $g$ satisfies the contact conditions and $V$ is horizontal, then the third term vanishes. That is the calculus reads the same way as it does for the standard complex partial derivatives in the plane.

Now let us consider CR-pluriharmonicity. Following \cite{DragTomCR}, we recall that a real-valued function $u \in C^\infty (\Om)$, where $\Om \subset  \Hei$ is open, is said to be CR-pluriharmonic, if
for any point $p \in \Om$, there is an open neighborhood $\Om'$ of $p$, and a real valued function $v \in C^\infty (\Om')$, such that $\bar Z(u + iv) = 0$ on $\Om'$. The following lemma characterises CR-pluriharmonicity.

%
%

\begin{lem}[\cite{DragTomCR}, Lemma 5.2]  A real-valued function $u \in C^\infty (\Om)$, where $\Om \subset  \Hei$ is open, is CR-pluriharmonic if and only if
	\begin{align}
Z^2 \bar Z u=0 \label{pluricond}
	\end{align}
for all $(z,t) \in \Om$. \end{lem}
%

\begin{cor}\label{lem53}
 If $f \in SU(1,2)$  then $\phi_f=\frac{1}{2} \ln \lambda_f $ is CR-pluriharmonic. Furthermore, this fact coupled with conformal rigidity, implies that the CR-pluriharmonicity of $\phi_f$ is a vacuous condition in the proposed CR-definition of a M\"obius transformation.
\end{cor}
\begin{proof}
 Recall the decomposition of $SU(1,2)$ maps in~\eqref{types12}. If $f$ is type 1 then $\lambda_f=r^2$ where $r$ is the dilation factor and so $\phi_f$ is trivially CR-pluriharmonic. If $f$ is type 2 then we have $f=\tau_p \circ \iota \circ g$ where $g=r_\theta  \circ \delta_r \circ \tau_q$ is type 1 and
 $$
 \phi_f = \phi_\iota \circ g+ \phi_g.
 $$
 It follows that
  $$
  Z^2 \bar Z \phi_f=Z^2 \bar Z (\phi_\iota \circ g )=Z^2 \bar Z \phi_\iota \circ g \,  ZG |ZG|^2.
  $$
  Moreover, $Z^2 \bar Z \phi_\iota=0$ by direct calculation.
\end{proof}

Next we look at the CR-Schwarzian in coordinates.

\begin{lem}[CR-Schwarzian, cf.~\cite{DuongNgocSon}]
 Let $f:\Om \subseteq \Hei \to \Hei$ be a positively oriented contact map and let $\phi=\frac{1}{2} \ln \lambda_f$, then it follows that
 \begin{align}
 B_\theta ( \phi) &= 2(Z^2\phi - 2 (Z\phi)^2  ) dz \otimes dz  +  2( \overline{Z^2  \phi -2( Z \phi)^2} ) d \bar z \otimes d \bar z. \label{Sf-CR}
 \end{align}
\end{lem}
\begin{proof} For $V,W \in T^{1,0} \Hei \oplus T^{0,1}  \Hei $, we have by definition, that $$\nabla^2 \phi(V, W )=VW\phi -d \phi(\nabla_V W) $$ and it follows that
$$ \nabla^2 \phi = (Z^2\phi) dz \otimes dz + (Z \overline{Z} \phi) dz \otimes d \bar z + ( \overline{Z} Z \phi) d \bar z \otimes dz+ ( \overline{Z}^2  \phi) d \bar z \otimes d \bar z$$ since $\nabla$ is flat, see \cite{DragTomCR} page 34.
Furthermore, we also have that $\partial_\flat \phi  = (Z\phi) dz$,  $\bar \partial_\flat \phi  = (\bar Z \phi) d \bar z$, and hence, with the notation of~\cite{DuongNgocSon}, the sub-Laplacian of $\phi$ reads
$$
- ( \Delta_\flat \phi) d \theta  (V, J W)=-(Z \overline{Z} \phi + \overline{Z} Z \phi ) (dz \otimes d \bar z+d \bar z \otimes dz)(V,W).
$$
Inserting these expressions into \eqref{CRSchwarz} results in \eqref{Sf-CR}.

\end{proof}

%


	The vanishing of the CR Schwarzian tensor is therefore equivalent to $Z^2\phi - 2 (Z\phi)^2=0$. Since
	 $$
	 \phi=\frac{1}{2} \ln \lambda_f, \quad  Z \phi =\frac{1}{2}\lambda_f^{-1} Z \lambda_f , \quad {\rm and} \quad  Z^2 \phi =\frac{1}{2} \lambda_f^{-2} \left ( \lambda_f Z^2 \lambda_f - (Z \lambda_f)^2 \right ),
	 $$
	we also have that
	\begin{align}
	Z^2 \phi  -2(Z\phi)^2 &= \frac{ 1}{2}\lambda_f^{-2} \left (  \lambda_f Z^2 \lambda_f  -  2(Z \lambda_f)^2 \right )=\frac{1}{2} \lambda_{f } Z^2(\lambda_f^{-1}). \label{lambdafDEF}
	\end{align}

Based on this discussion we propose the following definition serving as an analogue of \eqref{CplainDef}.

\begin{defn}\label{def-cr-schw}
 Let $f:\Om\to\Hei$ be a positively oriented contact mapping defined on an open set $\Om\subset \Hei$. The \emph{CR Schwarzian derivative} of $f$ is defined as follows:
 $$
 S_{CR}(f)=\frac{1}{2} \lambda_{f } Z^2(\lambda_f^{-1}).
 $$
\end{defn}	
Observe the following properties of $S_{CR}$:
\begin{itemize}
\item[(1)] If $f$ is conformal, then $S_{CR}(f)=0$ by case 8) in \eqref{confvfieldlist}. Furthermore,  the CR Schwarzian tensor vanishes as well.
\item[(2)] The condition $S_{CR}(f)=0$ is weak in the sense that the set of solutions contains an infinite dimensional vector space. In particular, the space of contact vector fields for which the flow $f_s$ satisfies $S_{CR}(f_s)=0$ is an infinite dimensional vector space as illustrated by the following examples.
 \end{itemize}

\begin{ex}\label{ex1}
 If $V=Yv_0 \, X - Xv_0\, Y -4 v_0 \, T$ and $v_0(x,y,t)=h(x)$ where $h$ is $C^2$, then the flow of $V$ has the form
\begin{align}
f_s(z,t) = \big (  z -i s h'(x),\, t +s( 2 x h'(x)- 4 h(x) ) \, \big ) \label{FlowEx}
\end{align}
and $\lambda_{f_s} =1$. Hence the CR Schwarzian derivative of $f_s$ vanishes for all $s\geq 0$.
\end{ex}

\begin{ex}\label{ex1SL2R}
Let $f$ be the linear action of  $\alpha \in SL(2,\R)$ , i.e.,
$$
f(z,t)=(\alpha(z),t)= (ax+by+i(cx+dy),t), \quad ad-bc=1,\,\, a,b,c,d\in \R, \, \, z=x+iy.
$$
The map $f$ is always a contact mapping, satisfies $\lambda_f=1$ and $S_{CR}(f)=0$, but is conformal only when $\alpha$ is a rotation.
\end{ex}

We discuss further properties of the CR-Schwarzian derivative and show that even if the contact mappings $f$ and $g$ satisfy $S_{CR}(f)=S_{CR}(g)=0$, it need not imply that $S_{CR}(f \circ g)=0$, see Claim 3 below.

First, we need an auxiliary result.

 \begin{observ}[Chain rule for the CR-Schwarzian] \label{obs-CR-Schw-chainr}
 If $f=(F=f_1+i f_2,f_3)$ and $g=(G=g_1+i g_2,g_3)$ are contact diffeomorphism between domains in $\Hei$, then the chain rule for $S_{CR}$ reads as follows:
 \begin{align} S_{CR}(f\circ g) & =  S_{CR}(f) \circ g \, \, (ZG)^2  \, + \, \overline{S_{CR}(f)} \circ g \, \,  (Z \bar G)^2  \, + \,S_{CR}(g) \nonumber \\
 & \quad + \Big (   \lambda_{f }  ( \bar ZZ\lambda_{f }  +  Z\bar Z \lambda_{f }  )  - 4 (Z\lambda_{f } )  (\bar Z \lambda_{f } )    \Big ) \circ g \,  \, \frac{ZG Z \bar G }{2 \lambda_f^2 \circ g }\nonumber \\
 & \quad + \frac{1}{2 \lambda_g }  \Big (   Z^2 G \lambda_g  -  2ZG Z\lambda_g  \Big )  Z (\ln( \lambda_f)) \circ g \nonumber \\
 & \quad +  \frac{1}{2 \lambda_g } \Big (  Z^2 \bar G   \lambda_g -  2 Z \bar G    Z\lambda_g     \Big  )   \bar Z (\ln( \lambda_f)) \circ g. \label{CR-Schw-chainr}
 \end{align}
 \end{observ}
\begin{proof} The result follows by direct calculation using the fact that  $ \lambda_{f\circ g} = \lambda_f \circ g \, \, \lambda_g$.
 \end{proof}

 Let us analyze some consequences of this chain rule.

 {\bf Claim 1.} \emph{If $f$ is contact and $g$ is conformal, then
 \begin{align}
  S_{CR}(f\circ g) & =  S_{CR}(f) \circ g \, \, (ZG)^2.  \label{rightcocycle}
 \end{align}}

 \emph{Proof:} If $g$ is conformal, then $\bar Z G$, $Z \bar G$ and $S_{CR}(g)$ are all zero, hence~\eqref{CR-Schw-chainr} reads as
 \begin{align*}S_{CR}(f\circ g) & =  S_{CR}(f) \circ g \, \, (ZG)^2   + \frac{1}{2 \lambda_g }  \Big (   Z^2 G \lambda_g  -  2ZG Z\lambda_g  \Big )  Z (\ln( \lambda_f)) \circ g.
 \end{align*}
 Furthermore, conformality of $g$ implies that $ Z^2 G \lambda_g  -  2ZG Z\lambda_g  =0$.  Indeed, cases 5), 6) and 8)  given at \eqref{confvfieldlist}, show that $Z^2 (\lambda_g^{-1}) =0$ and $ Z^2(\lambda_g^{-1}G)=0$. Calculating directly we get
 \begin{align*}
 Z(\lambda_g^{-1}G) = Z(\lambda_g^{-1}) G + \lambda_g^{-1} ZG \quad \hbox{ and }  \quad Z^2(\lambda_g^{-1}G) =  2Z(\lambda_g^{-1}) ZG + \lambda_g^{-1} Z^2G=  -2\lambda_g^{-2} Z\lambda_g ZG + \lambda_g^{-1} Z^2G.
\end{align*}
 Multiplying the previous equation by $\lambda_g^{2}$ shows that  $ Z^2 G \lambda_g  -  2ZG Z\lambda_g  =0$ and we get the right hand cocycle condition for $S_{CR}(f)$ as in the assertion of the claim.

 {\bf Claim 2. } \emph{If $f$ is conformal of type 1 (free of inversion, see Lemma~\ref{conftypes}) and $g$ is contact, then
 $$
 S_{CR}(f \circ g)= S_{CR}(g).
 $$  }
 \emph{Proof:} For a conformal $f$ and a contact $g$ it follows from~\eqref{CR-Schw-chainr} that
 \begin{align*} S_{CR}(f\circ g) & = S_{CR}(g) + \Big (   \lambda_{f }  ( \bar ZZ\lambda_{f }  +  Z\bar Z \lambda_{f }  )  - 4 (Z\lambda_{f } )  (\bar Z \lambda_{f } )    \Big ) \circ g \, \, \frac{ZG Z \bar G }{2 \lambda_f^2 \circ g }\\
 & \quad + \frac{1}{2 \lambda_g }  \Big (   Z^2 G \lambda_g  - 2ZG Z\lambda_g  \Big )  Z (\ln( \lambda_f)) \circ g\\
 & \quad +  \frac{1}{2 \lambda_g } \Big (  Z^2 \bar G   \lambda_g -  2 Z \bar G    Z\lambda_g    \Big  )   \bar Z (\ln( \lambda_f)) \circ g.
 \end{align*}
 Since $f$ is assumed to be of type 1 (free of inversion), then $\lambda_f=r^2$ where $r$ is the dilation factor. In this case the right hand side above reduces to $S_{CR}(g)$ and we obtain the claim.

 {\bf Claim 3. } \emph{If $f$ is inversion and $g$ is the linear action of $ \alpha \in SL(2,\R)$, then
 \[
 S_{CR}(f\circ g)=-6 \frac{ |G|^2 }{ \Vert  g \Vert_{\Hei}^{4} }    \,  ZG Z \bar G .
 \]
 In particular, $S_{CR}(f)=S_{CR}(g)=0$ but $S_{CR}(f \circ g)\not \equiv 0$.
 }

 \emph{Proof:} If $f$ is the inversion and $g$ is the linear action of $ \alpha \in SL(2,\R)$, then $\lambda_g=1$, $S_{CR}(g)=0$, $Z^2 G=Z^2 \bar G=0$ and by~\eqref{CR-Schw-chainr} and Lemma~\ref{lem53} we get
 \begin{align*}
  S_{CR}(f\circ g) & = \Big (   \lambda_{f }  ( \bar ZZ\lambda_{f }  +  Z\bar Z \lambda_{f }  )  - 4 (Z\lambda_{f } )  (\bar Z \lambda_{f } )    \Big ) \circ g \, \, \frac{ZG Z \bar G }{2 \lambda_f^2 \circ g } = -6 \frac{ |G|^2 }{ \Vert  g \Vert_{\Hei}^{4} }    \,  ZG Z \bar G .
 \end{align*}

\section{A classical type Schwarzian}

We now turn our attention to the classical type definition of the Schwarzian, denoted by $S_{CL}$, which although may look naive and only based on aesthetics, does in fact have some properties that are Schwarzian like. Moreover, the classical type Schwarzian differs from the CR Schwarzian derivative. Among other results, below we address the following properties of $S_{CL}$:
\begin{itemize}
 \item[(1)] $S_{CL}$ vanishes on $SU(1,2)$ and satisfies cocycle conditions, see Theorem~\ref{thm61} and Lemma~\ref{mycocycle};
 \item[(2)] in contrast with the CR Schwarzian derivative, we can construct contact flows which are zeroes of the CR Schwarzian derivative but nonzero for the classical type Schwarzian, see Example~\ref{ex3};
 \item[(3)] $S_{CL}$ is conformally invariant, see Corollary~\ref{cor63};
 \item[(4)]  the kernel for $S_{CL}$ contains not only the conformal mappings, see Proposition~\ref{prop66}.
\end{itemize}

\begin{defn}[Classical type Schwarzian]\label{defn-cl-Schw}
  Let  $\Omega \subseteq \mathbb{H}^1$ be an open subset, and $f: \Omega \to \mathbb{H}^1$ be a mapping satisfying the contact equations~\eqref{contact-eqs}. We write $f=(f_1+if_2,f_3)=(F,f_3)$ and define \emph{the classical type Schwarzian derivative of $f$} as follows
\begin{align}
S_{CL}(f):= Z \left (\frac{ Z^2 F}{ZF} \right )-\frac{1}{2} \left(\frac{Z^2F}{ZF}\right)^2= \frac{ Z^3 F}{ZF}-\frac{3}{2} \left(\frac{Z^2F}{ZF}\right)^2 \label{HeisSchwartz}
\end{align}
at points where $ZF\not=0$.
\end{defn}

The above definition is formulated without assuming that $f$ is a local diffeomorphism in order to give a slightly wider perspective. However, in what follows we will assume that $f$ is a contact transformation, cf. Definition~\ref{def-contact}.

\begin{thm}\label{thm61}
If $f$ is a conformal map of an open set in $\Hei$, then $S_{CL}(f)=0$.
\end{thm}

\begin{proof}  From cases 5), 6) and 8) given at \eqref{confvfieldlist}, we have that $Z^2 (\lambda_f^{-1}) =0$ and $ Z^2(\lambda_f^{-1}F)=0$. Calculating directly we get
	\begin{align*}
	Z(\lambda_f^{-1}F)= Z(\lambda_f^{-1}) F + \lambda_f^{-1} ZF \,\hbox{ and }\,Z^2(\lambda_f^{-1}F) =  2Z(\lambda_f^{-1}) ZF + \lambda_f^{-1} Z^2F=0.
	 \end{align*}

	It then follows that
	\begin{align} \frac{Z^2F}{ZF}&=-2  \lambda_f Z(\lambda_f^{-1}) = 2  \lambda_f^{-1} Z\lambda_f \label{ZZsqDivZ0}\\
	Z \left ( \frac{Z^2F}{ZF} \right ) &= 2 Z( \lambda_f^{-1} Z\lambda_f)  = 2  \left (  -\lambda_f^{-2} (Z\lambda_f)^2  +   \lambda_f^{-1} Z^2\lambda_f \right ).  \label{ZZsqDivZ}
	\end{align}
	From \eqref{ZZsqDivZ0} and \eqref{ZZsqDivZ} we have
	\begin{align}
	Z \left ( \frac{Z^2F}{ZF} \right )-\frac{1}{2} \left( \frac{Z^2F}{ZF} \right )^2 &= 2  \left (  -\lambda_f^{-2} (Z\lambda_f)^2  +   \lambda_f^{-1} Z^2\lambda_f \right ) -2 \lambda_f^{-2} (Z\lambda_f)^2  \nonumber \\
	&=   2 \lambda_f^{-2} \left ( \lambda_f Z^2\lambda_f - 2   (Z\lambda_f)^2  \right ) \label{Sconfcont} \\
	&=  2 \lambda_f^{-2} \lambda_f^{3} Z^2(\lambda_f^{-1}) \nonumber \\
	&=  2 \lambda_f Z^2(\lambda_f^{-1})=0. \nonumber
	\end{align}
\end{proof}

The condition $S_{CL}(f)=0$ does not imply that $f$ is conformal, as is demonstrated by the following example.
\begin{ex}\label{ex3}
The contact flow at~\eqref{FlowEx}  in Example~\ref{ex1} does not satisfy $S_{CL}(f_s)=0$ for all $h$, as for these flows, it can be shown directly that the $S_{CL}(f_s)=0$  implies that $h(x)=ax^2+bx+c$. Indeed, by direct calculation, the contact flow generated by $V= v_1{X} +v_2 Y-4v_0{T}$ satisfies
$$ \frac{d}{ds} S_{CL}(f_s)|_{s=0} =-2i Z^3 \bar Z v_0$$ and so $Z^3 \bar Z v_0=0$
is a necessary condition for $S_{CL}(f_s)=0$. Hence, if as in Example~\ref{ex1}, $v_0(x,y,t)=h(x)$, then $h(x)$ must be a polynomial of degree at most $3$ and in this case we can explicitly compute the flow $f_s$ and $S_{CL}(f_s)$. The condition $S_{CL}(f_s)=0$ can then be seen to require that the order of $h(x)$ be at most $2$.  For $h(x)=ax^2+bx+c$, the flow is given by
\begin{align*}
f_s(z,t) &= \big ( \, x+i(y-s(2ax+b))  \, , \, t-2s( bx + 2c) \, \big ).
\end{align*} On the other hand, if $h(x)=e^x$, then direct calculation with aid of MAPLE gives
$$
S_{CL}(f_s)=\frac{1}{8}  \frac{ (e^x s)^2 (272+104 e^{2x} s^2+17 e^{4x} s^4)}{(16+8 e^{2x} s^2+e^{4x} s^4)^2}+ \frac{i}{8} \frac{ e^x s (-256-32 e^{2x} s^2+8 e^{4x} s^4+4 e^{6x} s^6)}{(16+8 e^{2x} s^2+e^{4x} s^4)^2},
$$ where as $S_{CR}(f_s)=0$.

\end{ex}


The following lemma establishes the right cocycle condition for the classical type Schwarzian.
\begin{lem}\label{mycocycle}
	Let $f=(F,f_3): \Omega \to \mathbb{H}^1$ be a positively oriented contact mapping and let $g=(G,g_3) : \Omega' \to \Omega$ be a conformal mapping, then we have the cocycle condition:
	\begin{align}
	 S_{CL}(f \circ g)= S_{CL}(f)\circ g \, (ZG)^2 + S_{CL}(g). \label{cocycle}
	\end{align}	
\end{lem}

\begin{proof}
	First note that
	\begin{align}
	Z(F \circ g) = g^*dF(Z) &=g^*( ZF dz + \bar Z F  d \bar z +TF \theta)(Z) \nonumber \\
	&= \Big (  (ZF \circ g) \, dG + (\bar Z F \circ g) \, d \bar G + (TF \circ g) \, g^*\theta \Big )(Z) \nonumber \\
	&= (ZF \circ g)\, ZG + (\bar Z F \circ g) \, Z \bar G + (TF \circ g) \, g^*\theta(Z) \label{zfg1}
	\end{align}
	and the assumptions imply that
	\begin{align}
	Z(F \circ g)&= (ZF \circ g) ZG, \label{zfg2}
	\end{align}
	since  $g^*\theta(Z)=0$ and $Z\bar G= \overline{ \bar Z G}=0$. Differentiating \eqref{zfg2} with $Z$ and $Z^2$, and repeating the calculation at \eqref{zfg1} with $F$ replaced by $ZF$ and $Z^2F$ we obtain the following identities
	\begin{align*}
	Z^2(F \circ g) &=  ( Z^2F \circ g)(ZG)^2 + (ZF \circ g) Z^2G,\\
	Z^3(F \circ g) &= ( Z^3F \circ g)(ZG)^3  + 3 (Z^2F \circ g) Z^2G\, ZG+ (ZF \circ g) Z^3G.
	\end{align*}
	Hence we have
	\begin{align*}
	\frac{Z^3(F \circ g) }{Z(F \circ g)}  &=  \frac{( Z^3F \circ g)}{(ZF \circ g) }(ZG)^2  + 3 \frac{ (Z^2F \circ g)}{(ZF \circ g) } Z^2G + \frac{Z^3G}{ ZG}
	\end{align*}
	and
	\begin{align*}
	\frac{3}{2}\left ( \frac{Z^2(F \circ g)}{Z(F \circ g)} \right )^2&=    \frac{3}{2} \left ( \frac{ Z^2F \circ g }{ZF \circ g }\right )^2 (ZG)^2 +   3 \left ( \frac{Z^2F \circ g }{ZF \circ g } \right )  Z^2G  +  \frac{3}{2} \left (\frac{ Z^2G}{ ZG} \right )^2
	\end{align*}
	which imply assertion~\eqref{cocycle}, as then it holds that
	\begin{align*}
	S_{CL}(f \circ g) &= \left ( \frac{( Z^3F \circ g)}{(ZF \circ g) } -\frac{3}{2} \left ( \frac{ Z^2F \circ g }{ZF \circ g }\right )^2  \right ) (ZG)^2 + \frac{Z^3G}{ ZG}   -  \frac{3}{2} \left (\frac{ Z^2G}{ ZG} \right )^2.
	\end{align*}
\end{proof}

%

Consequently the $(2,0)$ form $\Sigma(f)=S_{CL}(f) dz \otimes dz$ is conformally invariant.
\begin{cor}\label{cor63}
 If $f$ is a local diffeomorphism that satisfies the contact conditions, and $(z,t)=g(w,s)=(G(w,s),g_3(w,s))$ where $g$ is conformal, then
	\begin{align*}
	S_{CL}(f \circ g)(w,s) dw \otimes dw & = S_{CL}(f)(z,t)  dz \otimes dz.
	\end{align*}
\end{cor}

Finally we examine the left cocycle condition.
\begin{lem}\label{mycocycle2}
	Let $\Om, \Om'\subset \Hei$ be domains and $f=(F,f_3): \Omega \to \mathbb{H}^1$ be a contact mapping. If $g=(G,g_3) : \Omega' \to \Omega$ is a conformal mapping, 
	then we have:
	\begin{align}  S_{CL}(g \circ f)= S_{CL}(f) & + \Big (  \frac{3}{2}  (\bar Z Z^2G \circ f)  -  3 \left ( \frac{Z^2G \circ f} {ZG \circ f }  \right ) \left (  \bar Z ZG \circ f    \right ) \, \Big ) \frac{(ZF) \,( Z \bar F)}{ZG \circ f}  \nonumber\\
	&   +  \left (\frac{\bar Z ZG \circ f} {ZG \circ f} \right ) \,  \Big (  (Z^2 \bar F)(ZF)    - 2  (Z^2F)( Z \bar F) \Big ) \frac{1}{ZF}\nonumber\\
	&  - \frac{3}{2}  \left (  \frac{\bar Z ZG \circ f }{ZG \circ f }   \right )^2   \,( Z \bar F)^2. \label{cocycle2} \end{align}	
\end{lem}

\begin{proof}
	First note that
	\begin{align}
	Z(G \circ f) = f^*dG(Z) &=f^*( ZG dz +TG \theta)(Z) \nonumber \\
	&= \Big (  (ZG \circ f) \, dF  + (TG \circ f) \, f^*\theta \Big )(Z) \nonumber \\
	&= (ZG \circ f)\, ZF  
	\end{align}
	since  $f^*\theta(Z)=0$ and $Z\bar G= \overline{ \bar Z G}=0$. Similarly, we also have that
	\begin{align*}
	Z^2(G \circ f) & = Z(ZG \circ f) ZF+ (ZG \circ f) Z^2F\\
	 & = f^* \Big ( ZZGdz +\bar Z ZG d \bar z +T ZG \theta \Big )   (Z)  ZF+ (ZG \circ f) Z^2F\\
	 & =  (Z^2G \circ f) ZF^2 + (\bar Z ZG\circ f) Z \bar F Z F  + (ZG \circ f) Z^2F\end{align*} and
	 \begin{align*}
	Z^3(G \circ f) & = (Z^3G \circ f ZF + \bar Z Z^2G \circ f Z \bar F)  ZF^2 + (Z^2G \circ f) 2 ZF \, Z^2F\\
	& \quad  +( \frac{1}{2} \bar Z Z^2 G\circ f ZF+ \bar Z^2 ZG \circ f Z \bar F) Z \bar F Z F  +(\bar Z ZG \circ f) (Z^2 \bar F Z F + Z \bar F Z^2 F) \\
	& \quad + (Z^2G \circ f ZF + \bar Z ZG \circ f Z \bar F ) Z^2F + (ZG \circ f) Z^3F,
	\end{align*}
	where in the second line of the above formula, we have used the fact that conformality of $g$ implies that
	$$Z \bar Z Z G=\frac{1}{2} \bar Z Z^2 G.$$
	Formula \eqref{cocycle2} follows by direct computation using the expressions above.
\end{proof}
Recall types 1 and 2 in Lemma~\ref{conftypes}.
\begin{cor} Let $f:\Om\to \Hei$ be a contact transformation. Then we have:

(1) If $g$ is type 1, then  $S_{CL}(g \circ f)= S_{CL}(f)$.

(2) If $g$ is type 2, then  $S_{CL}(g \circ f)= S_{CL}(f)$ if and only $f$ is conformal.
\end{cor}

\begin{proof}  If $g$ is type 1, then $Z ZG=0$ and $\bar Z ZG=0$. Hence the right hand side ot \eqref{cocycle2} reduces to $S_{CL}(f)$.

If $g$ is type 2, then $\bar Z ZG=0$ but $Z ZG \ne 0$ and $ \bar Z Z ZG \ne 0$. Then it follows that the right hand side of \eqref{cocycle2} reduces to
\begin{align*}
  S_{CL}(g \circ f)= S_{CL}(f) & +   \frac{3}{2}  (\bar Z ZZG \circ f)      \frac{(ZF) \,( Z \bar F)}{ZG \circ f}.
\end{align*}
Therefore, $S_{CL}(g \circ f)= S_{CL}(f)$ implies $\bar Z F=0$.

Conversely, if $f$ is conformal then $S_{CL}(f)=0$, $Z \bar F= \overline{\bar Z F}=0$, and the right hand side of \eqref{cocycle2} reduces to $0$. Hence we have  $S_{CL}(g \circ f)= S_{CL}(f)=0$.
\end{proof}

Next we study the size of the kernel for $S_{CL}$ and show that it contains not only the conformal mappings.

\begin{prop}\label{prop66}
	There exist a family of transformations $f$, parameterised by an infinite dimensional vector space, that satisfy the contact conditions and $S_{CL}(f)=0$.
\end{prop}
\begin{proof}
	If $f=(F, f_3)$ is a positively oriented contact diffeomorphism of a domain $\Omega \subseteq \Hei$, then  $ZF \ne 0$ on $\Omega$, and $S_{CL}(f)=0$ is equivalent to the system
	\begin{align}
	Zu-\frac{1}{2} u^2=0,  \quad u=\frac{ Z^2 F}{ZF}. \label{zzFzF}
	\end{align}
Moreover, the first equation can be written in the equivalent form $Z(-2/u)=1$. Conversely, we can start by solving the equation $Z(H)=1$ on $\Omega$ and observe that since $Z$ is analytically hypoelliptic, any solution $H$ must be analytic. For example, the function $H=h_1+i h_2$ where
\begin{align*}
h_1(x,y,t) &= C_1 (t + 2 xy) +2 x+ C_2 + \int \dbd{\psi}{x} dy\\
h_2(x,y,t) &= \psi(x,y) + C_3\\
\frac{ \partial^2 \psi}{\partial x^2} +\frac{ \partial^2 \psi}{\partial y^2}  &=  - 4C_1
\end{align*} solves $ZH=1$ on the domain of $\psi$. Moreover we can choose
$$ \psi(x,y)= Q(x,y)  -C_1(x^2+y^2) $$
where $Q$ is any real valued harmonic function, i.e., the choices for $Q$ form an infinite dimensional vector space.

Having found $H$ we then solve
\begin{align} Z^2 F+\frac{2}{H} ZF =0 \label{ZFord2}\end{align} on the support of $H$ with the condition $ZF \ne 0$. To this end, we can consider $W= \log ZF$ for some branch of the logarithm since $ZF \ne 0$ and \eqref{ZFord2} becomes  \begin{align} ZW=-\frac{2}{H}.  \label{ZWeq}\end{align}
Since the right hand side of \eqref{ZWeq} is analytic on the support of $H$, the Cauchy--Kovalevskaya theorem then implies the existence of an analytic solution $W$ defined on $\Omega$, and by the same argument, the equation $ZF=e^W$ has an analytic solution $F$ defined on $\Omega$. Furthermore, the Cauchy--Kovalevskaya theorem applies to \eqref{zcontact} for our given $F$, and so there is an analytic map $f=(F,f_3) : \Omega \to \Hei $ which satisfies the contact equations, although there is no guarantee that $f$ is a diffeomorphism.
\end{proof}	

\begin{rem}The map $f$ constructed above will be diffeomorphic on a neighborhood of any point $p$ at which $|ZF(p)| \ne |\bar Z(p)|$.
\end{rem}


\section{Preschwarzian and superharmonicity results for Jacobian and related differential expressions}

One of the main goals of this section is to introduce a notion of Preschwarzian in the setting of the Heisenberg group $\Hei$. Similar notions have been studied, for instance, in the context of planar harmonic mappings, see~\cite{HM} for the properties of the Preschwarzian. Moreover, in~\cite{HM} it is shown that several results known for Schwarzians have their natural counterparts for Preschwarzians, see the discussion preceding Theorem~\ref{thm-Pf} below. The fact that in the planar case, the Preschwarzian is a viable tool that is useful beyond the class of (quasi)conformal mappings, makes the Preschwarzian for various classes of mappings in $\Hei$ a worthy consideration. Since the key expression in Definition~\ref{def-presch} of the Preschwarzian is the horizontal Jacobian $J_F$, we focus our attention on studying its properties and the related expressions, such as $|ZF|^2$ and $|\nabla_H F|^2$. Namely, in Proposition~\ref{subh-lem2}, Lemma~\ref{subh-lem1} and Corollary~\ref{cor-subh}, we investigate super- and subharmonicity results for the class of harmonic mappings (Definition~\ref{harm-map}) and the so-called gradient harmonic mappings (Definition~\ref{grad-harm}), the latter one introduced here in $\Hei$. Our studies involve and emphasize the contact equations (without the assumption of local diffeomorphism, though), see Proposition~\ref{cont-f3}, relate gradient harmonic mappings to a system of non-homogeneous subelliptic harmonic equations, see~\eqref{eqs-Harmu}. Furthermore, we discover connections to the analysis of level sets, see~\eqref{geom-term} and the discussion therein, also the Bochner type identity in Lemma~\ref{lem-Bochner}. We finish presentation of results in this section with Theorem~\ref{thm-Pf} providing some growth estimates for a horizontal Jacobian of a gradient harmonic mapping. Similar results are known for harmonic mappings in the plane and their proofs are based on complex analysis and ODE techniques, whereas our proof relies on modern analysis involving the potential-theoretic properties and the upper gradients.


\begin{defn}\label{def-presch}
Let $\Om, \Om'\subset \Hei$ be domains and  let $f=(f_1,f_2,f_3):\Om\to\Om'$ be a $C^1$ mapping with a positive Jacobian $\lambda_f:=J_F>0$ in $\Om$, where $F=f_1+if_2$. The \emph{Preschwarzian of $f$ in $\Om$} is defined as follows:
\begin{equation}
Pf(x):=Z\ln J_F(x),\quad\hbox{for any } x\in\Om.
\end{equation}
\end{defn}

Notice that the coefficients of the CR-Schwarzian of $f$ can be expressed by $Pf$ as follows (cf.~\eqref{Sf-CR}):
\[
 2(Z^2\phi - 2 (Z\phi)^2)=ZPf-(Pf)^2,\quad \phi:=\frac12 \ln J_F.
\]
Therefore, a necessary condition for $S(f)\equiv 0$ is that $Pf\equiv 0$. Moreover, the following equation emphasizes further the link between the Jacobian and the Preschwarzian:
\begin{equation}
 |Pf|=|Z(\ln J_F)|=|\nabla_H \ln J_F|=\frac{1}{J_F}|\nabla_H J_F|.
\end{equation}

By direct computation, one also shows that
\[
 \Delta_H\ln J_F=8\Re(\bZ Pf).
\]

Observe that in Definition~\ref{def-presch} the third component function $f_3$ of $f$ plays no role and hence all mappings with the same third component function have the same Preschwarzian. However, later on we study mappings satisfying the contact equations and so their Preschwarzians carry information about all component functions.

The following lemma is a counterpart of Proposition 1 in \cite{HM} shown for sense preserving harmonic mapings in $\C$.
\begin{lem}[Composition properties of Preschwarzian]
 Let $f$ be as in Definition~\ref{def-presch}. If $g$ is conformal, then $P(f\circ g)=(Pf\circ g) ZG + Pg$. Moreover, if $A$ is affine, then $P(A\circ f)=Pf$.
\end{lem}
\begin{proof}
Let $V$ be any vector field with values in $\C \otimes T\Hei$ and $F$ a $C^1$ complex valued function. Let further $F$ and $g$ be local diffeomorphisms. Computations similar to those in the proof of Lemma~\ref{mycocycle} result in the following differentiation formula:
 $$
 V(F \circ g) = g^*dF(V)= (ZF \circ g) VG + (\bar Z F \circ g) V \bar G + (TF \circ g) g^*\theta(V).
 $$
 In particular, when $g$ is contact and $V \in \C \otimes T\mathcal{H}$, then
  $$
  V(F \circ g) =  (ZF \circ g) VG + (\bar Z F \circ g) V \bar G.
 $$
  When $g$ is conformal we have $Z(F \circ g) =  (ZF \circ g) ZG$ and $\bar Z(F \circ g) =  (\bar Z F \circ g) \overline{ZG}$. It follows that $\lambda_{f \circ g} = (\lambda_f \circ g ) \lambda_g$ and  $P(f \circ g) = Z \ln(\lambda_f \circ g )  + Pg$ since $\lambda_g=|ZG|^2$. The first assertion of the lemma now follows by observing that
 \begin{align*}
Z \ln(\lambda_f \circ g )& = \frac{Z(\lambda_{f} \circ g)}{\lambda_{f}\circ g} = \frac{ (Z\lambda_{f} \circ g) Zg}{\lambda_{f}\circ g}=(Pf\circ g) ZG.
 \end{align*}


Let $A(x):=(az+b\bar{z}+c, dt+e):=(F_A, A_3)$, where $a,b,c\in \C$ and $d,e\in \R$, be an affine map of $\Hei$. It follows that $ZF_A\equiv a$ which further implies that  $ZZF_A\equiv 0$ and $\bZ Z F_A\equiv 0$, moreover, $\bZ F_A \equiv b$. Therefore,
 \[
  P(A\circ f)=\frac{Z(\lambda_{A}\circ f)}{\lambda_{A}\circ f}+Pf=\frac{Z((|a|^2-|b|^2)\circ f)}{(|a|^2-|b|^2)\circ f}+Pf=Pf.
 \]
\end{proof}

 In order to compare the setting of Euclidean spaces and Heisenberg groups, also to illustrate rigidity implied by the group structure, we will now discuss some properties of mappings with equal Preschwarzians. In the setting of sense-preserving harmonic mappings in the plane, equal Schwarzians imply a homothety condition for sums of $z$-derivatives, see~\cite{CDO} and the discussion in~\cite[Section 5]{HM}. Furthermore, in the same setting, Theorem 3 in~\cite{HM} asserts that for simply-connected domains in $\C$ equal Preschwarzians imply the homothety condition for Jacobians of sense-preserving mappings. The proof of this result employs a number of standard tools in complex analysis, including the Riemann mapping theorem, a dilatation formulas for planar harmonic mappings and some basic ODE techniques, see pg. 79-81 in~\cite{HM}. Both, the assumptions imposed on mappings in subject and the complexity of that proof, remain in the striking contrast to the Heisenberg setting as the following observation shows.
\begin{prop}
 Let $f$ and $g$ be mappings in a domain $\Om\subset \Hei$ such that $J_F$ and $J_G$ have the same constant sign in $\Om$ and $P_f=P_g$ in $\Om$. Then $J_F=cJ_G$ for a constant $c>0$.
\end{prop}
 \begin{proof}
 The condition $P_f=P_g$ is equivalent to $(X-iY)\ln \frac{J_F}{J_G}=0$ and it follows that $\frac{J_F}{J_G}$ is annihilated by both $X$ and $Y$ on $\Om$. The bracket generating property then implies that $\frac{J_F}{J_G}$ is also annihilated by $T$ on $\Om$ and so we have that $\frac{\partial}{\partial t} \frac{J_F}{J_G}\equiv 0$, forcing  $\frac{\partial}{\partial x} \frac{J_F}{J_G}\equiv 0$ and $\frac{\partial}{\partial y} \frac{J_F}{J_G}\equiv 0$ to hold in $\Om$.
\end{proof}
Note that no contact conditions assumption on $f$ and $g$ are required in the above proof.

One of the key results for studying the potential-analytic properties of functions and mappings is the superharmonicity of the log of Jacobians. Below we discuss a variety of conditions implying superharmonicity and similar ones; also we relate such properties to the Preschwarzian. We begin with a preliminary technical result for sub-  and superharmonicity of mappings in $\Hei$. The general condition below is applied in Proposition~\ref{subh-lem2}.
\begin{lem}\label{subh-lem1}
 Let $\Om,\Om'\subset \Hei$ be domains and $f:\Om\to \Om'$ be a $C^3$-mapping. Then the following results hold:
 \begin{itemize}
 \item [(1)]  $\Delta_H \ln|ZF|^2\leq 0$ in $\Om\cap\{ZF\not=0\}$ provided that
\begin{equation}\label{cond-lem1}
\Re(ZF\,\Delta_H(\overline{ZF}))\leq 0\quad\hbox{ and }\quad \Re\left(\frac{\bZ \bar{F}}{ZF}\,ZZF\, \bZ ZF\right)\geq 0.
\end{equation}
The second condition can be equivalently formulated as $\Re\left(|ZF|^2\frac{ZZF}{ZF}\,\frac{\bZ ZF}{ZF}\right)\geq 0$.

\item[(2)] $\Delta_H |ZF|^2\geq 0$ provided that
\begin{equation}\label{cond2-lem1}
\Re(ZF\,\Delta_H(\overline{ZF}))\geq 0.
\end{equation}
\end{itemize}
\end{lem}
\begin{proof}
By direct computation we have
\[
 Z(\ln|ZF|^2)=\frac{1}{|ZF|^2}(ZZF\,\bZ\bar{F}+ZF\,Z\bZ\bar{F})
\]
and
\[
 \bZ Z(\ln|ZF|^2)=\frac{1}{|ZF|^4}(\bZ\bar{F}|ZF|^2\bZ ZZF+ZF|ZF|^2\bZ Z\bZ\bar{F}-(\bZ \bar{F})^2ZZF\,\bZ ZF-(ZF)^2Z\bZ\bar{F}\,\bZ\bZ\bar{F}).
\]
Therefore, by the assumptions~\eqref{cond-lem1}, we have
\begin{align*}
\Delta_H(\ln|ZF|^2)&=4(\bZ Z+Z\bZ)(\ln|ZF|^2)\\
&=\frac{1}{|ZF|^2}\left[\bZ \bar{F} \Delta_H(ZF)+ZF \Delta_H(\overline{ZF})-8\left(\frac{\bZ \bar{F}}{ZF}\,ZZF\, \bZ ZF+\overline{\frac{\bZ \bar{F}}{ZF}\,ZZF\, \bZ ZF}\right)\right]\\
&=\frac{2}{|ZF|^2}\left[\Re(ZF \Delta_H(\overline{ZF})-8\Re\left(\frac{\bZ \bar{F}}{ZF}\,ZZF\, \bZ ZF\right)\right]\leq 0.
\end{align*}
Similar computations imply that
\begin{equation*}
 \Delta_H |ZF|^2=4(\bZ Z+Z\bZ)|ZF|^2=8|ZZF|^2+8|\bZ ZF|^2+\frac14 \Re\left(ZF  \Delta_H(\overline{ZF})\right)\geq 0,
\end{equation*}
by assumption~\eqref{cond2-lem1}.
\end{proof}
\begin{rem}
 Observe that, by the above lemma, if the map $f$ satisfies $\bZ F=0$, then $J_F=|ZF|^2$ and we obtain an observation that $\ln J_F$ is superharmonic (outside the polar set of $J_F$) and $J_F$ is subharmonic.
\end{rem}

Next we recall some basic definitions in the theory of the Sobolev spaces in Heisenberg groups.

\begin{defn}\label{horiz-Sob}
 Let $\Om\subset \Hei$ be an open subset of the Heisenberg group $\Hei$. We say that a function $u:U\to \R$ belongs to the \emph{horizontal Sobolev space} $\HWtw(U)$ if $u\in L^2(U)$ and the horizontal derivatives $Xu$, $Yu$ exist in the distributional sense and are represented by elements of $L^2(U)$. The space $\HWtw(\Om)$ is a Banach space with respect to the norm
\[
 \|u\|_{\HW(U)}\,=\,\|u\|_{L^2(U)}+\|(Xu, Yu)\|_{L^2(U)}.
\]

In the similar way we define the local spaces $\HWtwloc (\Om)$ and Sobolev spaces of mappings $\HWtw(\Om, \Hei)$ and $\HWtwloc (\Om, \Hei)$ for mappings from $\Om$ to $\Hei$. We define space $\HWzero(\Om)$ as a closure of
$C_{0}^{\infty}(\Om)$ in $\HWtw(\Om)$.
\end{defn}

The horizontal gradient  $\nabla_H u$ of $u \in \HWtwloc (U)$ is the vector field given by
 $$
 \nabla_H u = (X u)X+(Y u)Y,
 $$
and $u$ is said to be {\it weakly harmonic/subelliptic harmonic} if
\begin{equation} \label{harm}
 \int_\Om \langle \nabla_H u, \nabla_H \phi \rangle dx =0
\end{equation} for all $\phi \in C_0^\infty(\Om)$.

\begin{defn}\label{harm-map}
We say that a map $f\in \HWtwloc(\Om, \Hei)$ is harmonic, if all its coordinate functions satisfy the subelliptic harmonic equations~\eqref{harm}. However, recall that a harmonic function in $\Hei$ is in fact analytic, see Section 5.10 in~\cite{blu}, meaning in particular that the strong (pointwise) Laplace equation holds for component functions of $f=(f_1, f_2, f_3)$:
\[
 \Delta_H f_i=0,\quad i=1,2,3.
\]
\end{defn}
For maps in the appropriate Sobolev-type spaces these are the Euler--Lagrange equations for the $2$-Dirichlet energy of the horizontal gradient $|\nabla_H f|$, see~\cite{wang}.

It turns out that imposing an additional assumption that $f$ satisfies contact equations (but $f$ is not necessarily a local diffeomorphism) we have the following observation.
\begin{prop}\label{cont-f3}
 Let $f\in \HWtwloc(\Om, \Hei)$ be a map satisfying contact equations~\eqref{contact-eqs} in the weak sense such that $\Delta_H f_1=\Delta_H f_2=0$. Then, also $\Delta_H f_3=0$.
\end{prop}
\begin{proof}
 Denote the coordinate functions of map $f$ as follows $f=(f_1, f_2, f_3)$. By the weak form of the contact equations, it follows that for any test function $\phi \in C_0^\infty(\Om)$, we have the following identities:
 \begin{equation*}
\int_{\Om} Xf_3\phi=\int_{\Om}2f_2Xf_1\phi-2f_1Xf_2\phi, \qquad
 \int_{\Om} Yf_3\phi=\int_{\Om} 2f_2Yf_1\phi-2f_1Yf_2\phi.
 \end{equation*}
 For any test function $\phi$, the functions $X\phi$ and $Y\phi$ are also test functions and so it follows that
 \begin{align*}
  \int_\Om (\Delta_Hf_3)\phi&=\int_\Om (X^2f_3+Y^2f_3)\phi\\
  &=- \int_\Om Xf_3X\phi - \int_\Om Yf_3Y\phi \\
  &=-\int_\Om 2f_2Xf_1 X\phi+\int_\Om 2f_1Xf_2 X\phi-\int_\Om 2f_2Yf_1 Y\phi+\int_\Om 2f_1Yf_2 XY\phi\\
  &=\int_\Om 2Xf_2Xf_1 \phi+2f_2X^2f_1 \phi-\int_\Om 2Xf_1Xf_2 \phi+2f_1X^2f_2 \phi\\
  &+\int_\Om 2Yf_2Yf_1 \phi+2f_2Y^2f_1 \phi-\int_\Om 2Yf_1Yf_2 \phi+2f_1Y^2f_2 \phi \\
 &=2\int_\Om f_2 (\Delta_Hf_1)\phi-f_1 (\Delta_H f_2)\phi=0,
 \end{align*}
 since $f_1$ and $f_2$ satisfy the strong (pointwise) Laplace equations. Thus $f_3$ is weakly- and hence also strongly harmonic in $\Om$.
\end{proof}

Another interesting class of mappings consists of those defined by the gradient of a subelliptic harmonic function. In the Euclidean setting, a mapping defined by the gradient of a harmonic function is also harmonic, meaning that each coordinate function satisfies the Laplace equation. Such a class of mappings plays a role in the studies of injectivity properties of harmonic maps, especially due to the Lewy Jacobian theorem~\cite{lew, glw}. The mappings in the class introduced below for the Heisenberg group need not satisfy the Laplace equation, nevertheless,  they satisfy the system of non-homogeneous harmonic equations, as well as the bi-Laplace type system of PDEs. Furthermore, such mappings are related to quasiconformal mappings, the Hessian determinant and to the studies of level sets, as  explained in Proposition~\ref{subh-lem2} below, and in the remark following it (see also Corollary~\ref{cor-subh} for related subharmonicity result).
\begin{defn}\label{grad-harm}
Let $\Om\subset \Hei$ be a domain and let $u:\Om\to\R$ be harmonic function in the sense that $\Delta_Hu=0$. Let $f:\Om\to \Hei$ be as follows:
\begin{equation}
 f:=(Xu, Yu, Tu).
\end{equation}
We call such an $f$ a \emph{gradient harmonic mapping} (or \emph{gradient harmonic} for simplicity). 
\end{defn}
If $f$ is gradient harmonic, then we set a notation $f=(F, f_3)$ for $F=\nabla_Hu$ and $f_3=Tu$ for a harmonic $u$. Observe that mapping $F$ is, in general, not contact.

As was mentioned above, it turns out that in contrast with the Euclidean setting, $F$ need not be harmonic map. Instead, by direct computations we have the following observation.
\begin{prop}
 Let $\Om\subset \Hei$ be a domain and  $f$ be a gradient harmonic map given by a harmonic function $u$. Then $f$ satisfies:
\begin{equation}\label{eqs-Harmu}
 \Delta_H f_1=8Tf_2,\quad \Delta_H f_2=-8Tf_1,\quad \Delta_H f_3=0.
\end{equation}
 In particular, $\Delta_H F=-8T\bar{F}$. Moreover, system~\eqref{eqs-Harmu} can be decoupled, as $f$ necessary satisfies the system of the bi-Laplace equations:
\begin{equation}\label{biLapl-eqs}
 \Delta_H(\Delta_Hf_1)=-64T^2f_1,\qquad \Delta_H(\Delta_Hf_2)=-64T^2f_2,\qquad \Delta_H f_3=0.
\end{equation}
\end{prop}
\begin{proof}
 We show the first equation in~\eqref{eqs-Harmu}, as the second one follows the same computations and the third is the immediate consequence of $T$ commuting with $\Delta_H$:
\begin{align*}
 \Delta_H f_1&=\Delta_H Xu=(XXXu+YYXu)\\
 &=X(-Y^2u)+Y(XYu-[X,Y]u)\\
 &=-XYYu+YXYu+Y(4Tu)\\
 &=-[X,Y](Yu)+4YTu\\
 &=8T(Yu)=8Tf_2.
\end{align*}
For the proof of~\eqref{biLapl-eqs} notice that by~\eqref{eqs-Harmu} we get
\[
 \Delta_H(\Delta_Hf_1)=8T(\Delta_Hf_2)=8T(-8Tf_1)=-64T^2f_1
\]
and so the proof is completed.
\end{proof}

\begin{rem}\label{Hess-gradh}
By direct computation one can show that the Jacobian $J_F$ of a gradient harmonic mapping $f=(F, f_3)$, and the determinant of the horizontal Hessian of the defining harmonic function $u$ are equal:
\[
 \det {\rm Hess}_Hu=J_F.
\]
Hence, $J_F$ carries information about the convexity properties and the geometry of the level sets for $f$ (cf. a remark following Proposition~\ref{subh-lem2}). Furthermore, this observation is employed to relate gradient harmonic mappings to quasiconformal mappings, see claim (2) in Proposition~\ref{subh-lem2}. We would like to point out that the literature concerned with notions of convexity in Heisenberg groups is wide, see e.g. \cite{br, gms, dgn}, where the symmetric Hessian ${\rm Hess}_H^*u$ is studied.  A relation between ${\rm Hess}_H^*u$ and a gradient harmonic $f$ is given by the following inequality
\[
 \det {\rm Hess}_H^*u=J_F-\frac14(Xf_2)^2-\frac14(Yf_1)^2+\frac12Xf_2Yf_1\leq J_F.
\]
\end{rem}

In next proposition we appeal to the notion of quasiconformal mappings, which we now briefly recall. The literature concerning quasiconformality and its generalizations is vast and we refer, for instance, to~\cite{kr1, kr2, pansu} for fundamental results for such mappings in the Heisenberg setting. Quasiconformal maps can be defined by three definitions, the analytic, the metric and the geometric one, all of which are equivalent on domains in $\mathbb{H}^1$. Here we follow the analytic definition.

  If $\Omega$ is an open set in $\Hei$, we say that a homeomorphism $f=(F, f_3) : \Omega \to  f(\Omega)\subset \Hei$ is \emph{$K$-quasiconformal} if $f\in HW^{1,4}_{loc}(\Om, \Hei)$ is weakly contact,  and  there exists $1 \leq K < \infty$ such that
\begin{equation*}
\|D_Hf (p)\|^4  \leq K J_F(p) \,\,\hbox{ for almost every }p\in\Om.
\end{equation*}
The Beltrami coefficient of $f$ is given by the expression $\mu_f(p)=\frac{\bZ F}{ZF}$ and we set $K_f(p):=\frac{1+|\mu_f|}{1-|\mu_f|}$ with convention that $K_f(p)=1$ at nonregular points of $f$. Then, any smooth contact mapping satisfying $1\leq K_f^2<\infty$ is quasiconformal.

The following proposition lists the superharmonicity properties of some differential expressions given by $F$, with $J_F$, the Jacobian of $F$ being the most important case.
\begin{prop}\label{subh-lem2}
 Let us consider a map $f$ between domains in $\Hei$. Then it holds that:
 \begin{itemize}
 \item[(1)] If $f$ is gradient harmonic, then $\Delta_H |ZF|^2\geq 0$ in $\Om$ and $\Delta_H \ln |ZF|^2\leq 0$ in $\Om\cap\{ZF\not=0\}$.
 \item[(2)] The gradient harmonic map $f$ is smooth quasiconformal if and only if $J_F>0$ uniformly in a domain of $f$.
 \item[(3)] Let $f\in C^1(\Om, \Hei)$ satisfy the contact equations~\eqref{contact-eqs} with $J_F\geq 0$ in $\Om$. If $f$ is a harmonic map, then the following are sufficient conditions for superharmonicity of $J_F$ in $\Om$ and $\ln J_F$ outside the polar set of $J_F$, i.e. on $\Om \cap \{J_F>0\} :$
\begin{align}
 & \nabla_H f_1\cdot \nabla_HTf_2 \leq \nabla_H f_2\cdot \nabla_HTf_1, \label{cond-subh-harm1} \\
 & |\nabla_H f_1|^2+|\nabla_H Tf_2|^2 \leq |\nabla_H f_2|^2+|\nabla_H Tf_1|^2, \label{cond-subh-harm2} \\
 & |\nabla_H f_1|\leq |\nabla_H f_2|\,\hbox{ and }\,|\nabla_H Tf_2|\leq |\nabla_H Tf_1|, \label{cond-subh-harm3} \\
 & |\nabla_H f_1|\leq |\nabla_H Tf_1|\,\hbox{ and }\,|\nabla_H Tf_2|\leq |\nabla_H f_2|. \label{cond-subh-harm4}
 \end{align}
 \item[(4)] Let $f\in C^1(\Om, \Hei)$ satisfy the contact equations~\eqref{contact-eqs} with $J_F\geq 0$ in $\Om$. If $f$ is gradient harmonic given by function $u$, then the following is a sufficient condition for superharmonicity of $J_F$ in $\Om$ and $\ln J_F$ outside the polar set of $J_F$:
 \[
\Delta_H(\nabla_Hu\cdot (YTu, -XTu))= \Delta_H(XuYTu-YuXTu)= \Delta_H(\nabla_Hu \cdot \star\nabla_HTu)\leq 0,
 \]
 where $\star$ stands for the Hodge star operator.
 \end{itemize}
 \end{prop}
Before proving the proposition, we remark that the expression $XuYTu-YuXTu$ appearing in the last claim above (see also Corollary~\ref{cor-subh} below) plays an important role in the studies of the level sets of functions in $\Hei$. Indeed, by~\cite[ Lemma 2.2]{FV} it follows that on $\{u=k\}\cap \{\nabla_H u\not=0\}$ we have
\begin{equation}\label{geom-term}
 XuYTu-YuXTu=-|\nabla_H u|^2 \langle Tn, v\rangle_H,
\end{equation}
where $n$ stands for the intrinsic normal vector to the level set $\{u=k\}$, and $v$ is the so-called intrinsic unit tangent direction to $\{u=k\}$, cf. the discussion of formulas (1) and (2) in~\cite{FV}. Hence, the subharmonicity property in claim (4) above can be further interpreted, for instance, in terms of the submean value property for the quantity $|\nabla_H u|^2 \langle Tn, v\rangle_H$. We further remark that~\eqref{geom-term} has also been investigated in the setting of the Grushin plane, where it turns out to be related to the convexity of the square of the function that locally parameterises the level sets, see~\cite[Remark 4.7]{FVg}.

\begin{proof}[Proof of Proposition~\ref{subh-lem2}]
Let map $f$ be a gradient harmonic given by harmonic function $u$. In order to prove claim (1) we use Lemma~\ref{subh-lem1} and verify conditions~\eqref{cond-lem1} for $\ln J_F$ and~\eqref{cond2-lem1} for $J_F$. We have $2\Delta_H(ZF)=\Delta_H(Xf_1+Yf_2)+i\Delta_H(Xf_2-Yf_1)$ and, since,
\[
 Xf_1+Yf_2=\Delta_Hu=0,\quad  Xf_2-Yf_1=-4Tu,
\]
it holds that $\Delta_H(ZF)=-2iT(\Delta_Hu)=0$. Thus, $\Delta_H(\overline{ZF})=\overline{\Delta_H(ZF)}=0$ and the first condition in~\eqref{cond-lem1} holds as the equality. Therefore also~\eqref{cond2-lem1} holds true showing the subharmonicity of $|ZF|^2$ for a gradient harmonic $f$. Moreover,
\begin{align*}
 ZZF&=\frac14(X-iY)(-4iTu)=-TYu-iTXu,\\
 \bZ ZF&=\frac14(X+iY)(-4iTu)=TYu-iTXu,\\
 \frac{\bar{Z}\bar{F}}{ZF}&=-1.
\end{align*}
Hence,
\[
  \Re\left(\frac{\bZ \bar{F}}{ZF}\,ZZF\, \bZ ZF\right)= \Re\left((TYu+iTXu)(TYu-iTXu)\right)=
  |\nabla_H(Tu)|^2\geq 0,
\]
 verifying the second condition in~\eqref{cond-lem1}.

Claim (2) is the consequence of the following equivalent conditions and Remark~\ref{Hess-gradh}, as the following equivalences hold for the Beltrami coefficient of $f$:
\begin{align*}
 \mu_f=\frac{|\bZ F|}{|ZF|}<1 &\Leftrightarrow (X^2u)^2+\frac14(XYu+YXu)^2<\frac14 (XYu-YXu)^2\\
 &\Leftrightarrow  (X^2u)(-Y^2u)+\frac14(2XYu)(2YXu)<0\\
  &\Leftrightarrow \det {\rm Hess}_Hu=J_F>0.
\end{align*}

We begin the proof of assertion (3) with an observation that in order to show superharmonicity of $\ln J_F$ it suffices to show that $J_F$ is superharmonic. Indeed,
\[
 \Delta_H (\ln J_F)=J_F^{-1}\Delta_H J_F- J_F^{-2}|\nabla_H J_F|^2\leq J_F^{-1}\Delta_H J_F.
\]
Next, recall from \eqref{contact-jac} that the contact equations imply that $J_F=Tf_3-2f_2Tf_1+2f_1Tf_2$
and therefore
\begin{align}
 \Delta_H J_F&=X(TXf_3-2Xf_2Tf_1-2f_2TXf_1+2Xf_1Tf_2+2f_1TXf_2) \nonumber\\
  &+Y(TYf_3-2Yf_2Tf_1-2f_2TYf_1+2Yf_1Tf_2+2f_1TYf_2)\nonumber\\
 &=T(\Delta_Hf_3)-2(\Delta_Hf_2)\,Tf_1-2f_2T(\Delta_Hf_1)
 +2(\Delta_Hf_1)Tf_2+2f_1T(\Delta_Hf_2)\nonumber\\
 &-4Xf_2TXf_1+4Xf_1TXf_2-4Yf_2TYf_1+4Yf_1TYf_2\nonumber\\
 &=T(\Delta_Hf_3)-2(\Delta_Hf_2)\,Tf_1-2f_2T(\Delta_Hf_1)
 +2(\Delta_Hf_1)Tf_2+2f_1T(\Delta_Hf_2)\nonumber\\
  &+4\nabla_H f_1\cdot \nabla_HTf_2 -4\nabla_H f_2\cdot \nabla_HTf_1. \label{eq-subh-cont}
\end{align}
Now with the additional assumption that $f$ is a harmonic map, \eqref{eq-subh-cont} reads as
\[
 \Delta_H J_F=4\nabla_H f_1\cdot \nabla_HTf_2 -4\nabla_H f_2\cdot \nabla_HTf_1\leq
 2|\nabla_H f_1|^2-2|\nabla_H f_2|^2+2|\nabla_H Tf_2|^2-2|\nabla_H Tf_1|^2,
\]
where in order to obtain the latter inequality we apply the Cauchy--Schwarz inequality. Hence, the assertions~\eqref{cond-subh-harm1}-\eqref{cond-subh-harm4} are straighforward.

In order to prove assertion (4) of the lemma, notice that by the assumptions of claim (4), we have that
\begin{equation}\label{prop-jac-cont}
 J_F=T^2u-2YuTXu+2XuTYu=T^2u+\nabla_Hu\cdot (YTu, -XTu)
\end{equation}
and the assertion follows immediately.
\end{proof}

Next we consider a counterpart of the following Bochner identity for harmonic functions in the setting of Riemannian manifolds:
\begin{equation*}
\frac12\Delta |\nabla u|^2=\|{\rm Hess}\,u\|^2+{\rm Ric}(\nabla u, \nabla u).
\end{equation*}
 In our observation, the Ricci-term is substituted by an expression whose geometric interpretation is explained in~\eqref{geom-term}.

\begin{lem}[The Bochner identity for subelliptic harmonic functions]\label{lem-Bochner}
 Let $u:\Om \to \R$ be a harmonic function in a domain $\Om\subset \Hei$. Then, the following Bochner-type identity holds pointwise:
 \begin{equation}\label{Bochner-H1}
  \frac12\Delta_H |\nabla_H u|^2=\|{\rm Hess}\,u\|^2+ \frac12 (XuYTu-YuXTu),
  \end{equation}
 where $\|{\rm Hess}\,u\|_{HS}$ stands for the Euclidean norm of the Hessian of $u$
 \[
  \|{\rm Hess}\,u\|^2:=(X^2u)^2+(XYu)^2+(YXu)^2+(Y^2u)^2.
 \]
\end{lem}
\begin{proof}[Proof of Lemma~\ref{lem-Bochner}]
By direct calculation we find that
\begin{align*}
 \Delta_H( (Xu)^2+(Yu)^2)&=2((X^2u)^2+(Y^2u)^2+(XYu)^2+(YXu)^2)+2Xu(X^3u+Y^2Xu)+2Yu(X^2Yu+Y^3u)\\
 &=\|{\rm Hess}\,u\|^2+2Xu(X^3u+Y^2Xu)+2Yu(X^2Yu+Y^3u).
\end{align*}
Next, observe we that harmonicity of $u$ implies that
\[
X^3u+Y^2Xu=-XY(Yu)+Y(YXu)=(-[X,Y]-YX)Yu+Y(XYu-[X,Y]u)=\frac12YTu
\]
and similarly $X^2Yu+Y^3u=-\frac12 XTu$. Inserting these two results into the first equation gives  assertion~\eqref{Bochner-H1}.
\end{proof}
The following observation provides conditions for subharmonicity of $|F|^2$ for gradient harmonic mappings. Since $|F|^2=|\nabla_H u|^2$ we  also obtain a result for subelliptic harmonic functions in $\Hei$. According to our best knowledge, this result is new in the literature for Heisenberg groups.
\begin{cor}\label{cor-subh}
 Let $f=(F, f_3):\Om\to \Hei$ be a gradient harmonic defined by a harmonic function $u:\Om\to \R$. Then $|F|^2=|\nabla_H u|^2$ is subharmonic, meaning that $\Delta_H |F|^2\geq 0$, provided that
 \begin{equation}\label{cond-subh}
  \nabla_H u\cdot \star \nabla_H Tu=Xu\,TYu-Yu\,TXu\geq 0.
 \end{equation}
\end{cor}
\begin{proof}
 Notice that under assumption~\eqref{cond-subh} it holds by~\eqref{eqs-Harmu} that
 \[
  \nabla_Hu\cdot (\Delta_H Xu, \Delta_H Yu)=(Xu, Yu)\cdot(8T(Yu),-8T(Xu))=8 \nabla_Hu \cdot \star \nabla_H Tu\geq 0,
 \]
 and so the assertion follows from~\eqref{Bochner-H1}.
\end{proof}

We finish our presentation with an application of the above discussion to the studies of the Jacobian for gradient harmonic mappings. Our motivation comes from the wide range of results for holomorphic functions in $\C$ and for planar harmonic mappings assuming the boundedness of Schwarzian or Preschwarzian norms on the unit ball $B$ with respect to weights depending on the distance to the boundary of the ball, for instance:
\[
\|Pf\|:=\sup_{x\in B}|Pf(x)|(1-|x|^2)^{2}\leq const.
\]
Such an estimate results in the upper and lower bounds for the Jacobian of $f$, also a growth estimates for $f$, see Theorems A, 2.4, 2.6, 2.7 and Proposition 2.8 in~\cite{HV}. Moreover, (pre)schwarzian norms play a fundamental role in the studies of univalency criterions for mappings, see~\cite{HM}, \cite{CO} and~\cite{CDO} and references therein.

Recall that $B=B(0,1)\subset \Hei$ means the unit open ball in the Kor\'anyi--Reimann distance $N$, i.e. $B=\{x\in \Hei: N(x)<1\}$.
\begin{theorem}\label{thm-Pf}
 Let $f:B \to \Hei$ be a gradient harmonic mapping satisfying contact equations~\eqref{contact-eqs} with $J_F\geq 0$ in $B$ given by a harmonic function $u$ satisfying $XuYTu-YuXTu\geq 0$ (cf. Corollary~\ref{cor-subh}). Suppose that
 \begin{equation}\label{thm-Pf-bound}
 \sup_{x\in B}|Pf(x)|(1-N(x)^4)^{\alpha}\leq c,
 \end{equation}
 for some $\alpha\geq 1$. Then, for all points $x, y \in B\setminus\{(z,t)\in B: z=0\}$ lying on a given radial horizontal curve with $N(x)>N(y)$ it holds that
 \[
  |J_F(x)-J_F(y)| \leq \frac{CN(x)}{(1-N(x))^{4+\alpha}}\|T^2u\|_{L^1(K_x)},
 \]
 where $K_x\Subset B$ is a compact set whose diameter depends on $N(x)$ only.
\end{theorem}

Let us comment that since $u$ is harmonic in $B$, the $L^1$-norm of $T^2u$ can be locally estimated in terms of $L^2$-norms of $u, |\nabla_H u|$ and $X^2u, Y^2u, XYu, YXu$, see e.g.~\cite{cap1, dom, xu} and their references.

\begin{proof}
 Since $u$ is harmonic, and hence, analytic on $\Hei$, the same holds for $J_F$, where $f$ is gradient harmonic. Therefore, $J_F\in \HWtwloc$ and $|\nabla_H J_F|$ is the weak upper gradient of $J_F$, see e.g. ~\cite[Chapter 11]{hk}. Hence, by the definition of the weak gradient, it follows that along any horizontal curve $\ga$ joining $y$ and $x\in B$ we have
$|J_F(x)-J_F(y)|\leq\int_{\ga} |\nabla_H J_F|$. However, since below we need to estimate $d(\ga, \partial B)$, we use the following explicit radial horizontal curves given by \cite[Section 3]{bat}
\begin{displaymath}
\gamma(r,(z,t)) = \left(rz e^{-\mathrm{i}\frac{t}{|z|^2}\log r},r^2 t\right),\quad (z,t)\in \partial B(0,1)\setminus \{z=0\}.
\end{displaymath}
It holds that $N(\gamma(r,(z,t)))=N( \delta_r(z,t) )$ and so by assumption~\eqref{thm-Pf-bound}, we obtain
 \begin{align}
 |J_F(x)-J_F(y)| &\leq \int_{\ga} |\nabla_H J_F| \nonumber \\
 &\leq \int_{\ga} \frac{c}{(1-N(w)^4)^{\alpha}} J_F(w)\leq \frac{c}{(1-N(x)^4)^{\alpha}} \int_{\ga} J_F(w). \label{est-Pf-bound}
\end{align}
The contact conditions for $f$, cf.~\eqref{prop-jac-cont}, together with the assumptions, imply that
\[
 J_F=T^2u-2(XuYTu-YuXTu)\leq T^2u.
\]
Since $u$ is harmonic so to is $T^2u$. Moreover, by the mean-value property, it now follows that for all points $y$ lying on radial curve $\ga$, and for balls $B(w, \frac12(1-N(x)))$, the following estimate holds:
\[
 T^2u(w)=\vint_{B(w,\frac12(1-N(x)))}T^2u(q)\leq \frac{16}{(1-N(x))^4|B(0,1)|}\|T^2u\|_{L^1(K_x)},
\]
where $K_x\Subset B(0,1)$ is a minimal compact set containing the balls $B(w, \frac12(1-N(x)))$ for all $w\in \ga$. Note here that the diameter of $K_x$  depends on $N(x)$ only since $K_x$ is a compact hull of balls of the form $B(w,\frac12(1-N(x)))$. Upon recalling that by quasiconvexity of the ball $B$, we have $\ell(\ga)\leq C N(x)$ where $C$ depends on the geometry of $\Hei$, and so we may complete estimate~\eqref{est-Pf-bound} as follows:
 \begin{align*}
 |J_F(x)-J_F(y)| & \leq \frac{c}{(1-N(x)^4)^{\alpha}} \int_{\ga} J_F(w) \\
 & \leq \frac{c}{(1-N(x)^4)^{\alpha}} \int_{\ga} \vint_{B(w,\frac12(1-N(x)))}T^2u(q) \\
 & \leq \frac{CN(x)}{(1-N(x))^4}\frac{1}{(1-N(x)^4)^{\alpha}} \|T^2u\|_{L^1(K_x)}.
\end{align*}
Since,
\[
 \frac{CN(x)}{(1-N(x))^4}\frac{1}{(1-N(x)^4)^{\alpha}}\leq \frac{CN(x)}{(1-N(x))^{4+\alpha}},
\]
the assertion follows.
\end{proof}

\section{Appendix: proof of Lemma~\ref{vzerosol}}

\begin{proof} We begin by establishing some fundamental commutation relations. First of all we have
	\begin{align*}
	0&=[Z, [ Z,\bar Z]]= Z [ Z,\bar Z]-  [ Z,\bar Z]Z = ZZ \bar Z -  2Z\bar Z Z  +  \bar Z Z Z,\\
	0&=[\bar Z, [ Z,\bar Z]]= \bar Z [ Z,\bar Z]-  [ Z,\bar Z] \bar Z= 2\bar Z Z \bar Z -  \bar Z\bar Z Z -  Z \bar Z \bar Z
	\end{align*}
	which imply
	\begin{align}
	2Z\bar Z Z &= ZZ \bar Z   +  \bar Z Z Z \label{Zsquared-2}\\
	2\bar Z Z \bar Z &=  \bar Z\bar Z Z +  Z \bar Z \bar Z. \label{Zsquared-1}
	\end{align}
	Moreover, multiplying \eqref{Zsquared-2} throughout by $\bar Z$ on the right and multiplying \eqref{Zsquared-1} throughout by $Z$ on the left, and then equating the results, gives \begin{align}
	\bar Z Z Z \bar  Z &=Z\bar Z \bar  Z Z. \label{Zsquared1}
	\end{align}
	Furthermore, since $2i T=[\bar Z ,Z]$ we have that
	\begin{align*}
	4 T^2 &= - \Big ( Z\bar Z Z\bar Z - Z\bar Z \bar Z Z - \bar Z Z Z\bar Z + \bar Z Z \bar Z Z \Big )\\
	&= - \Big ( Z ( Z \bar Z +2iT) \bar Z - Z\bar Z \bar Z Z - \bar Z Z Z\bar Z + \bar Z ( \bar Z  Z-2iT) Z \Big )\\
	&= - \Big ( Z Z \bar Z \bar Z - Z\bar Z \bar Z Z - \bar Z Z Z\bar Z + \bar Z  \bar Z  Z Z \Big ) + 2i (\bar Z Z  - Z \bar Z) T\\
	&= - \Big ( Z Z \bar Z \bar Z   - Z\bar Z \bar Z Z - \bar Z Z Z\bar Z + \bar Z  \bar Z  Z Z \Big ) - 4 T^2.
	\end{align*}
	Since $v_0$ is real valued, $Z^2v_0=0$ implies $\bar Z^2v_0=0$, and the previous identity together with \eqref{Zsquared1} implies that
	\begin{align}
	8T^2v_0 &= Z\bar Z \bar Z Zv_0 + \bar Z Z Z\bar Z v_0  \nonumber \\
	4T^2v_0 &=  Z\bar Z \bar Z Zv_0 =  \bar Z Z Z \bar Z v_0. \label{Zsquared2}
	\end{align}
	From \eqref{Zsquared-2} and  \eqref{Zsquared-1} we also have
	\begin{align}
	& ZZ \bar Z v_0 =  2Z\bar Z Z  v_0  \label{Zsquared3}\\
	&\bar Z \bar Z Z v_0 =  2   \bar Z  Z  \bar Z  v_0. \label{Zsquared4}
	\end{align}
	These identities together with \eqref{Zsquared1} give the following
	\begin{align}
	(\bar Z Z)^2  v_0 &= (Z  \bar Z)^2  v_0. \label{Zsquared5}
	\end{align}
	Using the two expression given by \eqref{Zsquared2} and $2i T=[\bar Z ,Z]$, we have
	\begin{align*}
	4T^3v_0 &= Z\bar Z \bar Z T Z v_0  &   4T^3v_0 &= \bar Z  Z  Z T \bar Z v_0 \\
	&= \frac{i}{2}Z\bar Z \bar Z Z \bar Z Z v_0   &  &=- \frac{i}{2} \bar Z  Z  Z \bar Z Z \bar Z v_0\\
	& = \frac{i}{2} \bar Z Z Z \bar Z \bar Z Z v_0 \quad \text{by } \eqref{Zsquared1} &  &=- \frac{i}{2}  Z \bar Z \bar Z  Z Z \bar Z v_0 \quad  \text{by }  \eqref{Zsquared1} \\
	& = \frac{i}{2} \bar Z Z  \bar Z  Z  Z \bar Z v_0 \quad  \text{by } \eqref{Zsquared1} &  &=- \frac{i}{2}  Z \bar Z  Z  \bar Z \bar Z  Z v_0 \quad  \text{by }  \eqref{Zsquared1} \\
	& = i \bar Z Z  \bar Z   Z\bar Z Z  v_0 = i(\bar Z Z )^3v_0 \quad  \text{by }  \eqref{Zsquared3} &  &= -i  Z \bar Z     Z \bar Z  Z  \bar Z  v_0   =-i( Z  \bar Z )^3v_0 \quad  \text{by }  \eqref{Zsquared4} .
	\end{align*}
	
	Hence
	\begin{align*}
	-(\bar Z Z)^3  v_0 &= (Z  \bar Z)^3  v_0 \\
	&= (Z  \bar Z) ( \bar Z  Z)^2  v_0 \quad \eqref{Zsquared5}\\
	&= Z  \bar Z \bar Z  Z \bar Z  Z v_0\\
	&= \bar Z   Z  Z  \bar Z \bar Z  Z v_0 \quad \eqref{Zsquared1}\\
	&= \bar Z   Z \bar Z   Z  Z \bar Z v_0 \quad \eqref{Zsquared1}\\
	&= 2\bar Z   Z \bar Z   Z  \bar Z  Z v_0=2(\bar Z Z)^3  v_0 \quad \eqref{Zsquared3} \\
	\end{align*}
	and we conclude that $(\bar Z Z)^3  v_0=0$ and so $T^3v_0=0$. Therefore, we write \begin{equation}
	v_0=a(x,y)t^2+b(x,y)t +c(x,y), \label{v-zero}
	\end{equation}
	an so the equation  $Z^2v_0=0$ can be split by taking real and imaginary parts and equating the coefficients of the powers of $t$ with $0$. The result is the following system of equations:
	\begin{align*}
	& {\frac {\partial ^{2}a}{\partial {x}^{2}}}  - {\frac {\partial ^{2}a}{\partial {y}^{2}}} =0  & &  {\frac {\partial ^{2}a}{\partial {x} \partial {y}}}=0\\
	& {\frac {\partial ^{2}b}{\partial {y}^{2}}}  -{\frac {\partial ^{2}b}{\partial {x}^{2}}} =   8 y{\frac {\partial a}{\partial x}} + 8\,x{\frac {\partial a}{\partial y}} & &  {\frac {\partial ^{2}b}{\partial x\partial y}} =  4\,x{\frac {\partial a}{\partial x}} -4 y {\frac {\partial a}{\partial y}} \\
	& {\frac {\partial ^{2}c}{\partial {y}^{2}}}  - {\frac {\partial ^{2}c}{\partial {x}^{2}}}= 4y{\frac {\partial b}{\partial x}} +4x{\frac {\partial b}{\partial y}} +8({y}^{2}-{x}^{2})a & & {\frac {\partial ^{2}c}{\partial x\partial y}} = 2x{\frac {\partial b}{\partial x}}  -2y{\frac {\partial b}{\partial y}} + 8\,xy a.
	\end{align*}
	Starting at the first row we integrate row by row as follows: Integrate the second equation to get expressions for the first partial derivatives and differentiate to get expressions for the nonmixed second order partial derivatives. Substitute these expressions into the first equation and resolve parameters.
	
	For example, the first row implies that
	$$a(x, y) = A_4(y^2+x^2) + A_3y + A_2 x + A_1. $$
	The updated second row is \begin{align}
	& {\frac {\partial ^{2}b}{\partial {y}^{2}}}  -{\frac {\partial ^{2}b}{\partial {x}^{2}}} =  32 A_4 xy + 8A_2 y + 8A_3 x & &  {\frac {\partial ^{2}b}{\partial x\partial y}} =  8A_4 (x^2-y^2)+ 4 A_2 x -  4A_3 y. \label{bequs}
	\end{align}
	The second equation above implies that
	\begin{align}  \frac {\partial b}{\partial y} &=-8 A_4 xy^{2} - 4A_3xy + \frac{8}{3} A_44 x^3 + 2 A_2 x^2 +  b_1( y ) \nonumber\\
	{\frac {\partial b}{\partial x}} & = 8 A_4 x^2y  + 4A_2 xy - \frac{8}{3} A_4  y^3 - 2 A_3 y^2   +  b_2( x ) \label{bequs2}
	\end{align} for some yet to be determined auxiliary functions $b_1(y)$ and $b_2(x)$.  Differentiating then gives
	\begin{align*}  \frac {\partial^2 b}{\partial y^2} &=-16 A_4 xy - 4A_3x + b_1'( y )\\
	{\frac {\partial^2 b}{\partial x^2}} & = 16 A_4 xy  + 4A_2y + b_2'( x )
	\end{align*} which when inserted into the first equation at \eqref{bequs} gives a consistency condition, namely $A_4=0$, moreover the result is the separated equation
	\begin{align*}
	&  b_1'( y ) -12A_2y  =  b_2'( x )+ 12A_3 x =B_1
	\end{align*}
	where $B_1$ is constant. Therefore
	\begin{align*}
	&  b_1'( y ) = 12A_2y  + B_1 \quad {\rm and} \quad  b_2'( x )=- 12A_3 x + B_1,
	\end{align*} which imply that
	\begin{align*}
	&  b_1( y ) = 6A_2y^2  + B_1y +B_2  \quad {\rm and} \quad b_2( x )=- 6A_3 x^2 + B_1x +B_3,
	\end{align*}
	and \eqref{bequs2} implies that
	\begin{align*}
	b &=  - 2A_3xy^2  + 2 A_2 x^2y + 2A_2y^3  + \frac{1}{2} B_1 y^2 +B_2 y +b_3(x)  \\
	b & =  2A_2 x^2 y - 2 A_3 xy^2 - 2A_3 x^3 +\frac{1}{2} B_1 x^2 +B_3 x +b_4(y) .
	\end{align*}
	Equating these expression for $b$ and separating variables shows that
	\begin{align*}
	- 2A_3 x^3 +\frac{1}{2} B_1 x^2 +B_3 x -b_3(x)  &=   2A_2y^3  + \frac{1}{2} B_1 y^2 +B_2 y -b_4(y) =B_4
	\end{align*} where $B_4$ is constant. Hence we have
	\begin{align*}
	b_3(x) = - 2A_3 x^3 +\frac{1}{2} B_1 x^2 +B_3 x -B_4   \quad {\rm and} \quad  b_4(y) &= 2A_2y^3  + \frac{1}{2} B_1 y^2 +B_2 y -B_4
	\end{align*}
	which imply that
	\begin{align*}
	b &= 2 A_2( y^3+ y x^2) -2A_3 ( xy^2  + x^3) + B_1(y^2+x^2) + B_2 y +B_4 x +B_3.
	\end{align*}
	Repeating the procedure on the third row of the original system yields $B_1=0$,  $A_2=0$, $A_3=0$ and
	$$c(x, y) = A_1(x^4+2 y^2 x^2+y^4)  - B_2 x(y^2  + x^2) + B_4 y(y^2+ x^2) + C_4 (y^2+x^2) +C_3 y + C_2x +C_1.$$
	The relabeling $$A_1=c_1, \, \, B_2=c_2, \, \,  B_3=c_7, \, \,  C_1=c_5,  \, \, C_2=c_8,  \, \, C_3=c_6, \, \,  C_4=c_4,  \, \, B_4=c_3$$
	results in the final form of $v_0$ as claimed. \end{proof}


\begin{thebibliography}{99}

\bibitem{acs} {\sc H. Arbel\'aez, M. Chuaqui, W. Sierra}, \emph{Nehari-type families of harmonic mappings.} Math. Nachr. 293 (2020), no. 1, 39--51.

\bibitem{br}{\sc Z. M. Balogh, M. Rickly}, \emph{Regularity of convex functions on Heisenberg groups}, Ann. Sc. Norm. Super. Pisa Cl. Sci. (5) 2 (2003), no. 4, 847--868.

\bibitem{bat} {\sc Z. Balogh, J. Tyson}, \emph{Polar coordinates in Carnot groups}, Math. Z. 241 (2002), no. 4, 697--730.

\bibitem{blu} {\sc A. Bonfiglioli, E. Lanconelli, F. Uguzzoni}, \emph{Stratified Lie Groups and Potential Theory for Their Sub-Laplacians}, Springer Monographs in Mathematics, Springer, 2007.

\bibitem{cap1} {\sc L. Capogna}, \emph{Regularity of quasi-linear equations in the Heisenberg group}, Comm. Pure Appl. Math. 50 (1997), no. 9, 867--889.

\bibitem{CapCow}{\sc L. Capogna, M. Cowling}, \emph{Conformality and q--harmonicity in carnot groups}, Duke Math. J. 135(3) (2006), 455--479.

\bibitem{CO} {\sc M. Chuaqui, B. Osgood}, \emph{Sharp distortion theorems associated with the Schwarzian derivative},
J. London Math. Soc. (2) 48 (1993), no. 2, 289--298.

\bibitem{CDO}{\sc M. Chuaqui, P. Duren, B. Osgood}, \emph{The Schwarzian derivative for harmonic mappings}, J. Anal.
Math. 91, 329--351 (2003).

\bibitem{CowLiOtazWu} {\sc M. G. Cowling, J. Li, A. Ottazzi and Q. Wu}, \emph{Conformal and CR mappings on Carnot groups}, Proc. Amer. Math. Soc. Ser. B 7 (2020), 67-81.

\bibitem{dgn} {\sc D. Danielli, N. Garofalo, D.-M. Nhieu}, \emph{Notions of convexity in Carnot groups}, Comm. Anal. Geom. 11 (2003), no. 2, 263--341.

\bibitem{dom} {\sc A. Domokos}, \emph{Differentiability of solutions for the non-degenerate $p$-Laplacian in the Heisenberg group}, J. Differential Equations 204 (2004), no. 2, 439--470.

\bibitem{DragTomCR} {\sc S. Dragomir, G. Tomassini}, \emph{Differential Geometry and Analysis on CR Manifolds},
Progress in Mathematics, 246. Birkh\"auser Boston, Inc., Boston, MA, 2006.

\bibitem{FV} {\sc F. Ferrari, Fausto, E. Valdinoci}, \emph{A geometric inequality in the Heisenberg group and its applications to stable solutions of semilinear problems}, Math. Ann. 343 (2009), no. 2, 351--370.

\bibitem{FVg} {\sc F. Ferrari, Fausto, E. Valdinoci}, \emph{Geometric PDEs in the Grushin plane: weighted inequalities and flatness of level sets} Int. Math. Res. Not. IMRN 2009, no. 22, 4232--4270.

\bibitem{glw} {\sc S. Gleason, T. Wolff}, \emph{Lewy's harmonic gradient maps in higher dimensions}, Comm. Partial Diff. Equations, 16, (1991), 1925--1968.


\bibitem{hk} {\sc P. Haj\l asz, P. Koskela}, \emph{Sobolev met Poincar\'e}, Mem. Amer. Math. Soc. 145 (2000), no. 688.

\bibitem{HM} {\sc R. Hern\'andez, M. Mart\'in}, \emph{Pre-Schwarzian and Schwarzian derivatives of harmonic mappings}, J. Geom. Anal. 25 (2015), no. 1, 64--91.

\bibitem{HV} {\sc R. Hern\'andez, O. Venegas}, \emph{Distortion theorems associated with Schwarzian derivative for harmonic mappings}, Complex Anal. Oper. Theory 13 (2019), no. 4, 1783--1793.



\bibitem{kr1} {\sc A. Kor\'anyi, H. M. Reimann}, \emph{Quasiconformal mappings on the Heisenberg group}, Invent. Math. 80(2) (1985),  309-338.


\bibitem{kr2} {\sc A. Kor\'anyi, H. M. Reimann}, \emph{Foundations for the theory of quasiconformal mappings on the Heisenberg group}, Adv. Math. 111(1) (1995), 1-87.


\bibitem{gms} {\sc G. Lu, J. Manfredi, B. Stroffolini}, \emph{Convex functions on the Heisenberg group}, Calc. Var. Partial Differential Equations 19 (2004), no. 1, 1--22.

\bibitem{lew} {\sc H. Lewy}, \emph{On the non-vanishing of the jacobian of a homeomorphism by harmonic gradients}, Ann. of Math. (2) 88 (1968), 518--529.





\bibitem{OsgStow} {\sc B. Osgood and D. Stowe}, \emph{The Schwarzian derivative and conformal mapping of Riemannian manifolds}, Duke Math. J. 67 (1992), 57--99.

\bibitem{VTS} {\sc V. Ovsienko, S. Tabachnikov}, \emph{What Is . . . the Schwarzian Derivative}, AMS Notices. 56 (1)  (2009), 34--36.


\bibitem{pansu}{\sc P. Pansu}, \emph{M\'etriques de Carnot-Carath\'eodory et quasiisom\'etries des espaces sym\'etriques de rang un},  Ann. of Math. (2) 129(1) (1989), 1--60.


\bibitem{Sato}{\sc H. Sato}, \emph{Schwarzian derivatives of contact diffeomorphisms},  Lobachevskii Journal of Mathematics, no. 4 (1999), 89--98.

\bibitem{DuongNgocSon} {\sc Duong Ngoc, Son}, \emph{The Schwarzian derivative and M\"obius equation on strictly pseudo-convex CR manifolds}, Commun. Anal. Geom. 26(2) (2018), 237--269.


\bibitem{xu} {\sc C.-J. Xu}, \emph{Regularity for quasilinear second-order subelliptic equations}, Comm. Pure Appl. Math. 45 (1992), no. 1, 77--96.

\bibitem{wang} {\sc Changyou Wang}, \emph{Subelliptic harmonic maps from Carnot groups}, Calc. Var. Partial Differential Equations 18(1) (2003), 95--115.


\end{thebibliography}
\end{document}